\documentclass{article}

\PassOptionsToPackage{round}{natbib}

\usepackage[final]{neurips_2020}

\usepackage[utf8]{inputenc} 
\usepackage[T1]{fontenc}    
\usepackage{hyperref}       
\usepackage{url}            
\usepackage{booktabs}       
\usepackage{amsfonts, amssymb, amsmath, amsthm}       
\usepackage{graphicx}
\usepackage{nicefrac}       
\usepackage{microtype}      

\usepackage{sidecap}

\usepackage{subfiles}
\usepackage{subcaption}

\usepackage{tikz}
\usetikzlibrary{arrows, positioning, calc}
\usetikzlibrary{decorations.pathreplacing, patterns}
\usepackage{mathtools}

\usepackage{color}
\definecolor{linkcolor}{RGB}{83,83,182}
\hypersetup{
    colorlinks=true,
    citecolor=linkcolor,
    linkcolor=linkcolor
}

\usepackage{thmtools, thm-restate}
\newtheorem{theorem}{Theorem}[section]
\newtheorem{lemma}[theorem]{Lemma}
\newtheorem{proposition}[theorem]{Proposition}
\newtheorem{conjecture}[theorem]{Conjecture}

\def\equationautorefname~#1\null{(#1)\null}

\usepackage{dsfont}
\newcommand{\ind}{\mathds{1}}  

\def\loss{\mathcal{L}}

\DeclareMathOperator*{\argmin}{arg\,min}
\DeclareMathOperator*{\sign}{sign}
\DeclareMathOperator{\Id}{Id}
\DeclareMathOperator{\ST}{ST}

\def\prox{\text{prox}}

\def\Rset{\mathbb R}

\def\psd{\emph{psd}}
\def\iid{\emph{i.i.d}}
\def\st{\emph{ s.t. }}
\def\ie{\emph{ i.e. }}
\def\wrt{\emph{w.r.t}}
\def\proxTV{prox-TV}

\newcommand{\changed}[1]{{{#1}}}

\title{Learning to solve TV regularized problems \\
       with unrolled algorithms}

\author{
  Hamza Cherkaoui\\
  Université Paris-Saclay, CEA, Inria \\
  Gif-sur-Yvette, 91190, France \\
  \texttt{hamza.cherkaoui@cea.fr} \\
\And
  Jeremias Sulam \\
  Johns Hopkins University \\
  \texttt{jsulam1@jhu.edu} \\
\And
  Thomas Moreau \\
  Université Paris-Saclay, Inria, CEA, \\
  Palaiseau, 91120, France \\
  \texttt{thomas.moreau@inria.fr} \\
}

\begin{document}

\maketitle

\begin{abstract}
Total Variation (TV) is a popular regularization strategy that promotes piece-wise constant signals by constraining the $\ell_1$-norm of the first order derivative of the estimated signal. The resulting optimization problem is usually solved using iterative algorithms such as proximal gradient descent\changed{, primal-dual algorithms or ADMM}. However, such methods can require a very large number of iterations to converge to a suitable solution. In this paper, we accelerate such iterative algorithms by unfolding proximal gradient descent solvers in order to learn their parameters \changed{for 1D TV regularized problems}. While this could be done using the \emph{synthesis} formulation, we demonstrate that this leads to slower performances. The main difficulty in applying such methods in the \emph{analysis} formulation lies in proposing a way to compute the derivatives through the proximal operator. \changed{As our main contribution, w}e develop and characterize two approaches to do so, describe their benefits and limitations, and discuss the regime where they can actually improve over iterative procedures. We validate those findings  \changed{with experiments on synthetic and real data}.
\end{abstract}

\section{Introduction}

Ill-posed inverse problems appear naturally in signal and image processing and machine learning, requiring extra regularization techniques.
Total Variation (TV) is a popular regularization strategy with a long history \citep{Rudin1992}, and has found a large number of applications in neuro-imaging \citep{Fikret2013}, medical imaging reconstruction \citep{Tian2011}, among myriad applications \citep{Rodriguez2013,Darbon2006}.
TV promotes piece-wise constant estimates by penalizing the $\ell_1$-norm of the first order derivative of the estimated signal, and it provides a simple, yet efficient regularization technique.

TV-regularized problems are typically convex, and so a wide variety of algorithms are in principle applicable.
Since the $\ell_1$ norm in the TV term is non-smooth, Proximal Gradient Descent (PGD) is the most popular choice \citep{Rockafellar1976}.
Yet, the computation for the corresponding proximal operator (denoted \proxTV{}) represents a major difficulty in this case as it does not have a closed-form analytic solution.
For 1D problems, it is possible to rely on dynamic programming to compute \proxTV{}, such as the taut string algorithm \citep{Davies2001,Condat2013}.
Another alternative consists in computing the proximal operator with iterative first order algorithm \citep{Chambolle2004, Beck2009, Boyd2011, Condat2013a}.
\changed{Other algorithms to solve TV-regularized problems rely on primal dual algorithms \citep{Chambolle2011, Condat2013a} or Alternating Direction Method of Multipliers (ADMM) \citep{Boyd2011}. These algorithms typically use one sequence of estimates for each term in the objective and try to make them as close as possible while minimizing the associated term. While these algorithms are efficient for denoising problems -- where one is mainly concerned with good reconstruction -- they can result in estimate that are not very well regularized if the two sequences are not close enough.}

When on fixed computational budget, iterative optimization methods can become impractical as they often require many iterations to give a satisfactory estimate.
To accelerate the resolution of these problems with a finite (and small) number of iterations, one can resort to unrolled and learned optimization algorithms (see \citealt{Monga2019} for a review).
In their seminal work, \citet{Gregor10} proposed the Learned ISTA (LISTA), where the parameters of an unfolded Iterative Shrinkage-Thresholding Algorithm (ISTA) are learned with gradient descent and back-propagation.
This allows to accelerate the approximate solution of a Lasso problem \citep{Tibshirani1996}, with a fixed number of iteration, for signals from a certain distribution.
The core principle behind the success of this approach is that the network parameters can adaptively leverage the sensing matrix structure \citep{Moreau2017} as well as the input distribution \citep{Giryes2016, Ablin2019}.
Many extensions of this original idea have been proposed to learn different algorithms \citep{Sprechmann2012, Sprechmann2013a, Borgerding2017} or for different classes of problem \citep{Xin2016, Giryes2016, Sulam2019}.
The motif in most of these adaptations is that all operations in the learned algorithms are either linear or separable, thus resulting in sub-differentials that are easy to compute and implement via back-propagation.
\changed{Algorithm unrolling is also used in the context of bi-level optimization problems such as hyper-parameter selection. Here, the unrolled architecture provides a way to compute the derivative of the inner optimization problem solution compared to another variable such as the regularisation parameter using back-propagation \citep{Bertrand2020}.}

The focus of this paper is to apply algorithm unrolling to TV-regularized problems \changed{in the 1D case}.
While one could indeed apply the LISTA approach directly to the \emph{synthesis} formulation of these problems, we show in this paper that using such formulation leads to slower iterative or learned algorithms compared to their \emph{analysis} counterparts.
The extension of learnable algorithms to the analysis formulation is not trivial, as the inner proximal operator does not have an analytical or separable expression.
We propose two architectures that can learn TV-solvers in their analysis form directly based on PGD.
The first architecture uses an exact algorithm to compute the \proxTV{} and we derive the formulation of its weak Jacobian in order to learn the network's parameters.
Our second method rely on a nested LISTA network in order to approximate the \proxTV{} itself in a differentiable way.
This latter approach can be linked to inexact proximal gradient methods \citep{Schmidt2011,Machart2012}.
These results are backed with numerical experiments on synthetic and real data.
\changed{Concurrently to our work, \citet{Lecouat2020} also proposed an approach to differentiate the solution of TV-regularized problems. While their work can be applied in the context of 2D signals, they rely on smoothing the regularization term using Moreau-Yosida regularization, which results in smoother estimates from theirs learned networks. In contrast, our work allows to compute sharper signals but can only be applied to 1D signals.}



The rest of the paper is organized as follows.
In \autoref{sec:iterative_algorithm}, we describe the different formulations for TV-regularized problems and their complexity.
We also recall central ideas of algorithm unfolding. \autoref{sec:prox_backprop} introduces our two approaches for learnable network architectures based on PGD.
Finally, the two proposed methods are evaluated on real and synthetic data in \autoref{sec:exp}.

\paragraph{Notations}
For a vector $x\in\Rset^k$, we denote $\|x\|_q$ its $\ell_q$-norm.
For a matrix $A\in\Rset^{m\times k}$, we denote $\|A\|_2$ its $\ell_2$-norm, which corresponds to its largest singular value and $A^\dagger$ denotes its pseudo-inverse.
For an ordered subset of indices $\mathcal S \subset \{1,\dots, k\}$, $x_{\mathcal S}$ denote the vector in $\Rset^{|\mathcal S|}$ with element $(x_{\mathcal S})_t = x_{i_t}$ for $i_t \in \mathcal S$.
For a matrix $A \in \Rset^{m\times k}$, $A_{:, \mathcal S}$ denotes the sub-matrix $[A_{:, i_1}, \dots A_{:, i_{|\mathcal S|}}]$ composed with the columns $A_{:, i_t}$ of index $i_t\in\mathcal S$ of $A$.
For the rest of the paper, we refer to the operators $D \in \Rset^{k-1\times k}, \widetilde D \in \Rset^{k\times k}, L \in \Rset^{k\times k}$ and $R \in \Rset^{k\times k}$ as:  \\[.5em]
{\centering
\begin{tabular}{ c c c c }
\scalebox{.7}{
$
D = \begin{bmatrix}
        -1 & 1 & 0 & \dots & 0\\
        0 & -1 & 1 & \ddots & \vdots\\
         \vdots & \ddots & \ddots & \ddots & 0 \\
        0  & \dots & 0 & -1 & 1 \\
    \end{bmatrix}
$
} &
\scalebox{.7}{
$
\widetilde D = \begin{bmatrix}
         1 & 0 & \dots &  0\\
        -1 & 1 & \ddots & \vdots\\
              & \ddots & \ddots & 0 \\
              & 0 & -1 & 1 \\
    \end{bmatrix}
$
} &
\scalebox{.7}{
$
L = \begin{bmatrix}
        1 & 0 & \dots & 0 \\
        1 & 1 & \ddots & \vdots\\
        \vdots & \ddots & \ddots & 0\\
        1 & \dots & 1 & 1\\
    \end{bmatrix}
$
} &
\scalebox{.7}{
$
R = \begin{bmatrix}
        0 & 0 & \dots & 0\\
        0 & 1 & \ddots & \vdots\\
        \vdots & \ddots & \ddots & 0\\
        0 & \dots & 0 & 1\\
    \end{bmatrix}
$
} \\
\end{tabular}
}

\section{Solving TV-regularized problems}
\label{sec:iterative_algorithm}

We begin by detailing the TV-regularized problem that will be the main focus of our work.
Consider a latent vector $u\in\Rset^k$, a design matrix $A \in \Rset^{m\times k}$ and the corresponding observation $x \in \Rset^m$.
The original formulation of the TV-regularized regression problem is referred to as the \emph{analysis} formulation \citep{Rudin1992}.
For a given regularization parameter $\lambda > 0$, it reads
\begin{equation}
    \label{eq:tv_analysis}
    \min_{u \in \Rset^k} P(u) = \frac12 \| x - Au\|_2^2 + \lambda \|u\|_{TV},
\end{equation}
where $\|u\|_{TV} = \|Du\|_1$, and $D \in \Rset^{k-1\times k}$ stands for the first order finite difference operator, as defined above.
The problem in \autoref{eq:tv_analysis} can be seen as a special case of a Generalized Lasso problem \citep{Tibshirani2011}; one in which the analysis operator is $D$.
Note that problem $P$ is convex, but the $TV$-norm is non-smooth.
In these cases, a practical alternative is the PGD, which iterates between a gradient descent step and the \proxTV{}. This algorithm's iterates read
\begin{align}
    \label{eq:analysis_tv_PGD}
    u^{(t+1)} = \prox_{\frac{\lambda}{\rho}\|\cdot\|_{TV}}\left(u^{(t)} - \frac{1}{\rho}A^\top(Au^{(t)} - x)\right)
    \enspace,
\end{align}
where $\rho = \|A\|_2^2$ and the \proxTV{} is defined as
\begin{equation}
    \label{eq:analysis_proximal_op}
    \prox_{\mu\|\cdot\|_{TV}}(y) = \argmin_{u \in \Rset^k} F_y(u) = \frac12 \|y - u\|_2^2 + \mu\|u\|_{TV}.
\end{equation}

Problem \autoref{eq:analysis_proximal_op} does not have a closed-form solution, and one needs to resort to iterative techniques to compute it.
In our case, as the problem is 1D, the \proxTV{} problem can be addressed with a dynamic programming approach, such as the taut-string algorithm \citep{Condat2013}.
This scales as $O(k)$ in all practical situations and is thus much more efficient than other optimization based iterative algorithms \citep{Rockafellar1976, Chambolle2004, Condat2013a} for which each iteration is $O(k^2)$ at best.

With a generic matrix $A\in\Rset^{m\times k}$, the PGD algorithm is known to have a sublinear convergence rate \citep{Combettes2011}.
More precisely, for any initialization $u^{(0)}$ and solution $u^*$, the iterates satisfy
\begin{equation}
    \label{eq:analysis_rate}
    P(u^{(t)}) - P(u^*) \le \frac{\rho}{2t}\|u^{(0)} - u^*\|_2^2,
\end{equation}
where $u^*$ is a solution of the problem in \autoref{eq:tv_analysis}. Note that the constant $\rho$ can have a significant effect. Indeed, it is clear from \autoref{eq:analysis_rate} that doubling $\rho$ leads to consider doubling the number of iterations.

\subsection{Synthesis formulation}
\label{sub:synthesis_tv}

An alternative formulation for TV-regularized problems relies on removing the analysis operator $D$ from the $\ell_1$-norm and translating it into a synthesis expression \citep{Elad2007}.
Removing $D$ from the non-smooth term simplifies the expression of the proximal operator by making it separable, as in the Lasso.
The operator $D$ is not directly invertible but keeping the first value of the vector $u$ allows for perfect reconstruction.
This motivates the definition of the operator $\widetilde D \in \Rset^{k\times k}$, and its inverse $L\in \Rset^{k\times k}$, as defined previously.
Naturally, $L$ is the discrete integration operator.
Considering the change of variable $z = \widetilde D u$, and using the operator $R \in \Rset^{k\times k}$, the problem in \autoref{eq:tv_analysis} is equivalent to
\begin{equation}
    \label{eq:tv_synthesis}
    \min_{z \in \Rset^{k}} S(z) = \frac 12 \|x - A L z\|_2^2 + \lambda\|R z\|_1.
\end{equation}
Note that for any $z \in \Rset^k$, $S(z) = P(Lz)$.
There is thus an exact equivalence between solutions from the synthesis and the analysis formulation, and the solution for the analysis can be obtained with $u^{*}=Lz^{*}$.
The benefit of this formulation is that the problem above now reduces to a Lasso problem \citep{Tibshirani1996}. In this case, the PGD algorithm is reduced to the ISTA with a closed-form proximal operator (the soft-thresholding).
Note that this simple formulation is only possible in 1D where the first order derivative space is unconstrained.
In larger dimensions, the derivative must be constrained to verify the Fubini's formula that enforces the symmetry of integration over dimensions.
While it is also possible to derive synthesis formulation in higher dimension \citep{Elad2007}, this does not lead to simplistic proximal operator.

For this synthesis formulation, with a generic matrix $A\in\Rset^{m\times k}$, the PGD algorithm has also a sublinear convergence rate \citep{Beck2009} such that
\begin{equation}
\label{eq:cvg_rate_synthesis}
    P(u^{(t)}) - P(u^*) \le \frac{2\widetilde\rho}{t}\|u^{(0)} - u^*\|_2^2,
\end{equation}
with $\widetilde\rho = \|AL\|_2^2$ (see \autoref{sub:proof:cvg_rate_synthesis} for full derivation).
While the rate of this algorithm is the same as in the analysis formulation  -- in $O(\frac 1t)$ -- the constant $\widetilde\rho$ related to the operator norm differs.
We now present two results that will characterize the value of $\widetilde\rho$.

\begin{restatable}{proposition}{speedLowerBound}[Lower bound for the ratio $\frac{\|AL\|^2_2}{\|A\|^2_2}$ expectation]
    \label{prop:lower_bound_AL}
    Let $A$ be a random matrix in $\Rset^{m\times k}$ with \iid{} \changed{normally distributed entries}. The expectation of $\|AL\|^2_2/\|A\|_2^2$ is asymptotically lower bounded when $k$ tends to $\infty$ by
    \[
        \mathbb E\left[\frac{\|AL\|_2^2}{\|A\|_2^2}\right] \ge
            \frac{2k + 1}{4\pi^2} + o(1)
    \]
\end{restatable}


\begin{figure}[t]
    \begin{minipage}{.55\textwidth}
        \includegraphics[width=\textwidth]{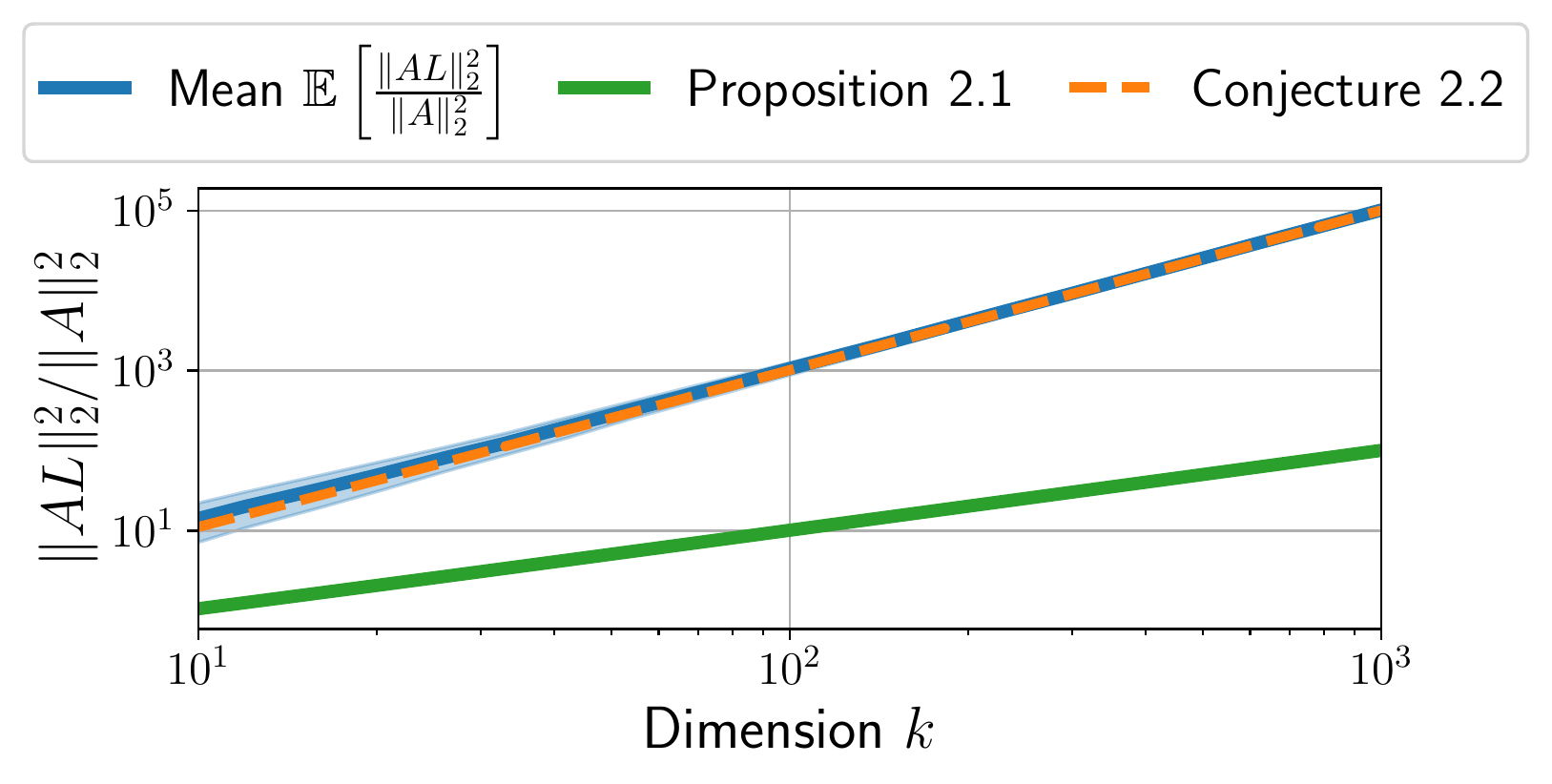}
    \end{minipage}%
    \begin{minipage}{.45\textwidth}
    	\caption{Evolution of $\mathbb E\Big[\frac{\|AL\|^2_2}{\|A\|^2_2}\Big]$ \wrt{} the dimension $k$ for random matrices $A$ with \iid{} \changed{normally distributed entries}. In light blue is the confidence interval [0.1, 0.9] computed with the quantiles. We observe that it scales as $O(k^2)$ and that our conjectured bound seems tight.}
      \label{fig:comparison_AL_A}
    \end{minipage}%
\end{figure}

The full proof can be found in \autoref{sub:proof:prop:lower_bound_AL}.
The lower bound is constructed by using \mbox{$A^TA \succeq \|A\|_2^2u_1u_1^\top$} for a unit vector $u_1$ and computing explicitely the expectation for rank one matrices.
To assess the tightness of this bound, we evaluated numerically $\mathbb E\left[\frac{\|AL\|^2_2}{\|A\|^2_2}\right]$ on a set of $1000$ matrices sampled with \iid{} \changed{normally distributed entries}. The results are displayed \wrt{} the dimension $k$ in \autoref{fig:comparison_AL_A}. It is clear that the lower bound from \autoref{prop:lower_bound_AL} is not tight. This is expected as we consider only the leading eigenvector of $A$ to derive it in the proof. The following conjecture gives a tighter bound.

\begin{conjecture}[Expectation for the ratio $\frac{\|AL\|^2_2}{\|A\|^2_2}$]
   \label{conj:expectation_AL_A}
    Under the same conditions as in \autoref{prop:lower_bound_AL}, the expectation of $\|AL\|_2^2 /\|A\|_2^2$ is given by
    \[
        \mathbb E\left[\frac{\|AL\|_2^2}{\|A\|_2^2}\right]
           = \frac{(2k+1)^2}{16\pi^2} + o(1)
           \enspace .
    \]
\end{conjecture}

We believe this conjecture can potentially be proven with analogous developments as those in \autoref{prop:lower_bound_AL}, but integrating over all dimensions.
However, a main difficulty lies in the fact that integration over all eigenvectors have to be carried out jointly as they are not independent.
This is subject of current ongoing work.

Finally, we can expect that $\widetilde \rho /\rho$ scales as $\Theta(k^2)$.
This leads to the observation that $\frac{\widetilde\rho}{2} \gg \rho$ in large enough dimension.
As a result, the analysis formulation should be much more efficient in terms of iterations than the synthesis formulation -- as long as the \proxTV can be dealt with efficiently.

\subsection{Unrolled iterative algorithms}
\label{sub:unrolled}

As shown by \citet{Gregor10}, ISTA is equivalent to a recurrent neural network (RNN) with a particular structure.
This observation can be generalized to PGD algorithms for any penalized least squares problem of the form
\begin{equation}
    \label{eq:general_problem}
    u^*(x) = \argmin_u \loss(x, u) = \frac12 \|x - Bu\|_2^2 + \lambda g(u)
    \enspace ,
\end{equation}
where $g$ is proper and convex, as depicted in \autoref{fig:network_pgd}.
By unrolling this architecture with $T$ layers, we obtain a network $\phi_{\Theta^{(T)}}(x) = u^{(T)}$ -- illustrated in \autoref{fig:lpgd_nn} -- with parameters \mbox{$\Theta^{(T)} = \{W_x^{(t)}, W_u^{(t)}, \mu^{(t)}\}_{t=1}^T$}, defined by the following recursion
\begin{equation}
    \label{eq:nn_recursion}
    u^{(0)} = B^\dagger x~;\qquad u^{(t)} = \prox_{\mu^{(t)}g}(W_x^{(t)}x + W_u^{(t)}u^{(t-1)})
    \enspace .
\end{equation}
As underlined by \autoref{eq:analysis_rate}, a good estimate $u^{(0)}$ is crucial in order to have a fast convergence toward $u^*(x)$.
However, this chosen initialization is mitigated by the first layer of the network which learns to set a good initial guess for $u^{(1)}$.
For a network with $T$ layers, one recovers exactly the $T$-th iteration of PGD if the weights are chosen constant equal to
\begin{equation} \label{eq:init_weights}
    W_x^{(t)} = \frac{1}{\rho}B^\top,
    \qquad W_u^{(t)} = (\Id - \frac{1}{\rho}B^\top B)~,
    \qquad \mu^{(t)} = \frac{\lambda}{\rho},
    \qquad \text{with} ~ \rho = \|B\|_2^2~
    \enspace.
\end{equation}
In practice, this choice of parameters are used as initialization for a posterior training stage.
In many practical applications, one is interested in minimizing the loss \autoref{eq:general_problem} for a fixed $B$ and a particular distribution over the space of $x$, $\mathcal P$.
As a result, the goal of this training stage is to find parameters $\Theta^{(T)}$ that minimize the risk, or expected loss, $\mathbb E[\loss(x, \phi_{\Theta^{(T)}}(x))]$ over $\mathcal P$.
Since one does not have access to this distribution, and following an empirical risk minimization approach with a given training set $\{x_1, \dots x_N\}$ (assumed sampled \iid{} from $\mathcal P$), the network is trained by minimizing
\begin{equation}
    \label{eq:training_loss}
    \min_{\Theta^{(T)}} \frac{1}{N}\sum_{i=1}^N  \loss(x_i, \phi_{\Theta^{(T)}}(x_i))
    \enspace.
\end{equation}
Note that when $T\to +\infty$, the presented initialization in \autoref{eq:init_weights} gives a global minimizer of the loss for all $x_i$, as the network converges to exact PGD.
When $T$ is fixed, however, the output of the network is not a minimizer of \autoref{eq:general_problem} in general.
Minimizing this empirical risk can therefore find a weight configuration that reduces the sub-optimality of the network relative to \autoref{eq:general_problem} over the input distribution used to train the network.
In such a way, the network learns an algorithm to approximate the solution of \autoref{eq:general_problem} for a particular class or distributions of signals.
\changed{It is important to note here that while this procedure can accelerate the resolution the problem, the learned algorithm will only be valid for inputs $x_i$ coming from the same input distribution $\mathcal P$ as the training samples. The algorithm might not converge for samples which are too different from the training set, unlike the iterative algorithm which is guaranteed to converge for any sample.}

This network architecture design can be directly applied to TV regularized problems if the synthesis formulation \autoref{eq:tv_synthesis} is used.
Indeed, in this case PGD reduces to the ISTA algorithm, with $B = AL$ and $\prox_{\mu g} = \ST(\cdot, \mu)$ becomes simply a soft-thresholding operator (which is only applied on the coordinates $\{2,\dots k\}$, following the definition of $R$).
However, as discussed in \autoref{prop:lower_bound_AL}, the conditioning of the synthesis problem makes the estimation of the solution slow, increasing the number of network layers needed to get a good estimate of the solution.
In the next section, we will extend these learning-based ideas directly to the analysis formulation by deriving a way to obtain exact and approximate expressions for the sub-differential of the non-separable \proxTV{}.

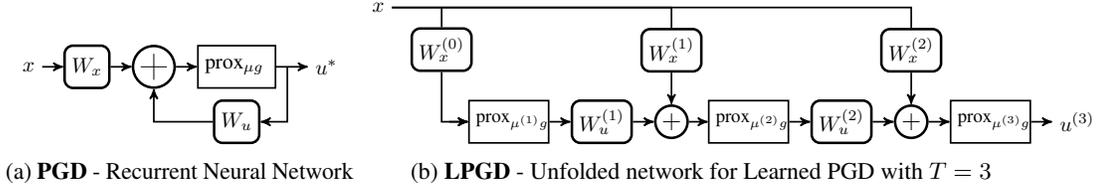
\begin{figure}[tp!]
    \begin{subfigure}[b]{.33\textwidth}
        \centering
        \scalebox{.8}{

\begin{tikzpicture}[scale=1]

\tikzset{
    >=stealth',
    varstyle/.style={
           rectangle,
           fill=white,
           rounded corners,
           draw=black, very thick,
           minimum width=2em,
           minimum height=2em,
           text centered},
    op/.style={
           circle,
           draw=black, very thick,
           fill=white,
           text width=1em,
           minimum height=1em,
           text centered},
    st/.style={
           draw,
           thick,
           rectangle, 
           fill=white,
           text width=3em, 
           minimum height=2em,},
    pil/.style={
           ->,
           thick,
    }
};

\node[varstyle] (linear) {$W_x $};
\node[op, inner sep=4pt, right=1em of linear] (add) {};
\node[left=1em of linear] (X) {$x$} edge[pil] (linear.west);
\path (linear.east) edge[pil] (add.west);


\draw ($ (add.north) - (0, .15)$) -- ($ (add.south) + (0, .15)$);
\draw ($ (add.east) - (.15, 0)$) -- ($ (add.west) + (.15, 0)$);

\node[st, right=1em of add] (h) {prox$_{\mu g}$};
\node[right=1.5em of h] (Z) {$u^*$};
\path (add.east)
    edge[pil] (h.west);
\path (h.east)
    edge[pil] (Z);
\node[inner sep=0,minimum size=0,right=.5em of h] (branch) {};
\node[varstyle,below=.5em of h] (S) {
    $W_u$
};
\draw[pil] (branch) |- (S.east);
\draw[pil] (S.west) -| (add);
\end{tikzpicture}
}
        \caption{\textbf{PGD} - Recurrent Neural Network}
        \label{fig:network_pgd}
    \end{subfigure}
    \begin{subfigure}[b]{.65\textwidth}
        \centering
        \scalebox{.8}{

\begin{tikzpicture}[scale=1]

\tikzset{
    >=stealth',
    varstyle/.style={
           rectangle,
           fill=white,
           rounded corners,
           draw=black, very thick,
           minimum width=2em,
           minimum height=2em,
           text centered},
    op/.style={
           circle,
           draw=black, very thick,
           fill=white,
           text width=.5em,
           minimum height=.3em,
           text centered},
    st/.style={
           draw,
           thick,
           rectangle, 
           fill=white,
           text width=3.1em, 
           minimum height=2em,
           text centered},
    pil/.style={
           ->,
           thick,
    }
};

\def\nlayer{2}

\node (X) {$x$};
\node[right=2em of X] (linear) {}; 
\node[below=4.7em of linear] (p) {};
\node[varstyle, above=2em of p] (B) {
        $W_x^{(0)}$
};

\foreach \i in {1,...,\nlayer}{
    \node[st, right=.9em of p] (sta\i) {
        \footnotesize prox$_{\mu^{(\i)}g}$
    };
    \path (p.east) edge[pil] (sta\i.west);

    \node[varstyle, right=1em of sta\i] (S) {
            $W_u^{(\i)}$
    };
    \node[op, inner sep=4pt, right=1em of S] (p) {};
    \node[varstyle, above=1.5em of p] (B\i) {
            $W_x^{(\i)}$
    };
    
    \draw ($ (p.north) - (0, .15)$) -- ($ (p.south) + (0, .15)$);
    \draw ($ (p.east) - (.15, 0)$) -- ($ (p.west) + (.15, 0)$);
    
    \path (sta\i.east) edge[pil] (S.west);
    \path (S.east) edge[pil] (p.west);
    \draw[pil] (X) -| (B\i) -- (p.north);
}
\draw[pil] (linear.center) -- (B) |- (sta1.west);

\node[st, right=1em of p] (sta) {
        \footnotesize prox$_{\mu^{(3)}g}$
};
    
\pgfmathtruncatemacro{\nlayer}{\nlayer + 1}
\node[right=1em of sta] (Z) {$u^{(\nlayer)}$};
\path (p.east) edge[pil] (sta.west);
\path (sta.east) edge[pil] (Z.west);

\end{tikzpicture}
}
        \caption{\textbf{LPGD} - Unfolded network for Learned PGD with $T=3$}
        \label{fig:lpgd_nn}
    \end{subfigure}
    \caption{
        \textbf{Algorithm Unrolling} - Neural network representation of iterative algorithms. The parameters $\Theta^{(t)} = \{W_x^{(t)}, W_u^{(t)}, \mu^{(t)}\}$ can be learned by minimizing the loss \autoref{eq:training_loss} to approximate good solution of \autoref{eq:general_problem} on average.}
\end{figure}


\section{Back-propagating through TV proximal operator}
\label{sec:prox_backprop}

Our two approaches to define learnable networks based on PGD for TV-regularized problems in the analysis formulation differ on the computation of the \proxTV{} and its derivatives.
Our first approach consists in directly computing the weak derivatives of the exact proximal operator while the second one uses a differentiable approximation.

\subsection{Derivative of \proxTV{}}
\label{sub:prox_jacobian}

While there is no analytic solution to the \proxTV{}, it can be computed exactly (numerically) for 1D problems using the taut-string algorithm \citep{Condat2013}.
This operator can thus be applied at each layer of the network, reproducing the architecture described in \autoref{fig:lpgd_nn}.
We define the LPGD-Taut network $\phi_{\Theta^{(T)}}(x)$ with the following recursion formula
\begin{equation}
    \phi_{\Theta^{(T)}}(x) =  \prox_{\mu^{(T)}\|\cdot\|_{TV}}
        \left(W_x^{(T)} x +  W_u^{(T)}\phi_{\Theta^{(T-1)}}(x)\right)
\end{equation}
To be able to learn the parameters through gradient descent, one needs to compute the derivatives of \autoref{eq:training_loss} \wrt{} the parameters $\Theta^{(T)}$. Denoting $h = W_x^{(t)} x +  W_u^{(t)}\phi_{\Theta^{(t-1)}}(x)$ and $u = \prox_{\mu^{(t)}\|\cdot\|_{TV}}(h)$, the application of the chain rule (as implemented efficiently by automatic differentiation) results in
\begin{equation}
    \frac{\partial \loss}{\partial h} = J_x(h, \mu^{(t)})^\top \frac{\partial \loss}{\partial u}~,
    \quad\text{ and }\quad
     \frac{\partial \loss}{\partial \mu^{(t)}} = J_\mu(h, \mu^{(t)})^\top \frac{\partial \loss}{\partial u}
     \enspace ,
\end{equation}
where $J_x(h, \mu) \in \Rset^{k\times k}$ and $J_\mu(h, \mu) \in \Rset^{k\times 1}$ denotes the weak Jacobian of the output of the proximal operator $u$ with respect to the first and second input respectively.
We now give the analytic formulation of these weak Jacobians in the following proposition.

\begin{restatable}{proposition}{jacobianProxTV}[Weak Jacobian of \proxTV{}]
    \label{prop:jacobian_prox_tv}
    Let $x\in\Rset^k$ and $u = \prox_{\mu\|\cdot\|_{TV}}(x)$, and denote by $\mathcal S$ the support of $z = \widetilde Du$.
    Then, the weak Jacobian $J_x$ and $J_\mu$ of the \proxTV{} relative to $x$ and $\mu$ can be computed as
    \begin{align*}
        J_x(x, \mu)
             = L_{:, \mathcal S}(L_{:, \mathcal S}^\top
                L_{:, \mathcal S})^{-1} L_{:, \mathcal S}^\top
    \quad \text{ and } \quad
        J_\mu(x, \mu)
            = - L_{:, \mathcal S}(L_{:, \mathcal S}^\top
              L_{:, \mathcal S})^{-1}\sign(Du)_{\mathcal S}
    \end{align*}
\end{restatable}

The proof of this proposition can be found in \autoref{sub:proof:prop:jacobian_prox_tv}.
Note that the dependency in the inputs is only through $\mathcal S$ and $\sign(Du)$, where $u$ is a short-hand for $\prox_{\mu\|\cdot\|_{TV}}(x)$.
As a result, computing these weak Jacobians can be done efficiently by simply storing $\sign(Du)$ as a mask, as it would be done for a RELU or the soft-thresholding activations, and requiring just $2(k-1)$ bits.
With these expressions, it is thus possible to compute gradient relatively to all parameters in the network, and employ them via back-propagation.


\subsection{Unrolled prox-TV}
\label{sub:lista_prox_unrolling}

As an alternative to the previous approach, we propose to use the LISTA network to approximate the \proxTV{} \autoref{eq:analysis_proximal_op}.
The \proxTV{} can be reformulated with a synthesis approach resulting in a Lasso \ie{}
\begin{equation}
    \label{eq:synth_prox_tv}
    z^* = \argmin_z \frac12 \|h - Lz\|_2^2 + \mu\|Rz\|_1
\end{equation}
The proximal operator solution can then be retrieved with $\prox_{\mu\|\cdot\|_{TV}}(h) = Lz^*$.
This problem can be solved using ISTA, and approximated efficiently with a LISTA network \citet{Gregor10}.
For the resulting architecture -- dubbed LPGD-LISTA -- $\prox_{\mu\|\cdot\|_{TV}}(h)$ is replaced by a nested LISTA network with a fixed number of layers $T_{in}$ defined recursively with $z^{(0)} = Dh$ and
\begin{equation}
   z^{(\ell+1)} =
        \ST\left(W_z^{(\ell, t)}z^{(\ell)}  + W_h^{(\ell, t)}\Phi_{\Theta^{(t)}},
                 \frac{\mu^{(\ell, t)}}{\rho}\right)
        \enspace .
\end{equation}
Here, $W_z^{(\ell, t)}, W_h^{(\ell, t)}, \mu^{(\ell, t)}$ are the weights of the nested LISTA network for layer $\ell$.
They are initialized with weights chosen as in \autoref{eq:init_weights} to ensure that the initial state approximates the \proxTV{}.
Note that the weigths of each of these inner layers are also learned through back-propagation during training.

The choice of this architecture provides a differentiable (approximate) proximal operator.
Indeed, the LISTA network is composed only of linear and soft-thresholding layers -- standard tools for deep-learning libraries.
The gradient of the network's parameters can thus be computed using classic automatic differentiation.
Moreover, if the inner network is not trained, the gradient computed with this method will converge toward the gradient computed using \autoref{prop:jacobian_prox_tv} as $T_{in}$ goes to $\infty$ (see \autoref{prop:cvg_jacobian}).
Thus, in this untrained setting with infinitely many inner layers, the network is equivalent to LPGD-Taut as the output of the layer also converges toward the exact proximal operator.


\paragraph{Connections to inexact PGD}
\label{sec:inexact_pgd}

A drawback of approximating the \proxTV{} via an iterative procedure is, precisely, that it is not exact.
This optimization error results from a trade-off between computational cost and convergence rate.
Using results from \citet{Machart2012}, one can compute the scaling of $T$ and $T_{in}$ to reach an error level of $\delta$ with an untrained network. \autoref{prop:T_in_lower_bound} shows that without learning, $T$ should scale as $O(\frac{1}{t})$ and $T_{in}$ should be larger than $O(\ln(\frac{1}{\delta}))$. This scaling gives potential guidelines to set these parameters, as one can expect that learning the parameters of the network would reduce these requirement.




\section{Experiments}
\label{sec:exp}


All experiments are performed in Python using \texttt{PyTorch} \citep{Paszke2019}.
We used the implementation\footnote{Available at \url{https://github.com/albarji/proxTV}} of \citet{Barbero2018} to compute TV proximal operator using taut-string algorithm.
The code to reproduce the figures is available online\footnote{Available at \url{https://github.com/hcherkaoui/carpet}.}.

In all experiments, we initialize $u_0 = A^\dagger x$.
Moreover, we employed a normalized $\lambda_{reg}$ as a penalty parameter: we first compute the value of $\lambda_{\max}$ (which is the minimal value for which $z=0$ is solution of \autoref{eq:tv_synthesis}) and we refer to $\lambda$ as the ratio so that $\lambda_{reg} = \lambda \lambda_{\max}$, with $\lambda \in [0, 1]$ (see \autoref{sub:lambda_max}).
As the computational complexity of all compared algorithms is the same except for the proximal operator, we compare them in term of iterations.

\subsection{Simulation}
\label{sub:simu}

We generate $n=2000$ times series and used half for training and other half for testing and comparing the different algorithms. \changed{We train all the network's parameters jointly -- those to approximate the gradient for each iteration along with those to define the inner proximal operator. The full training process is described in \autoref{sec:training_process}.}
We set the length of the source signals $(u_i)_{i=1}^{n} \in \Rset^{ n \times k}$ to $k=8$ with a support of $|S|=2$ non-zero coefficients (larger dimensions will be showcased in the real data application).
We generate $A \in \Rset^{m \times k}$ as a Gaussian matrix with $m=5$, obtaining then $(u_i)_{i=1}^{n} \in \Rset^{ n \times p}$. Moreover, we add Gaussian noise to measurements $x_i = Au_i$ with a signal to noise ratio (SNR) of $1.0$.

We compare our proposed methods, LPGD-Taut network and the LPGD-LISTA with $T_{in} = 50$ inner layers to PGD and Accelerated PGD with the analysis formulation. For completeness, we also add the FISTA algorithm for the synthesis formulation in order to illustrate \autoref{prop:lower_bound_AL} along with its learned version.

\begin{figure}[t]
    \centering
    \includegraphics[width=\columnwidth]{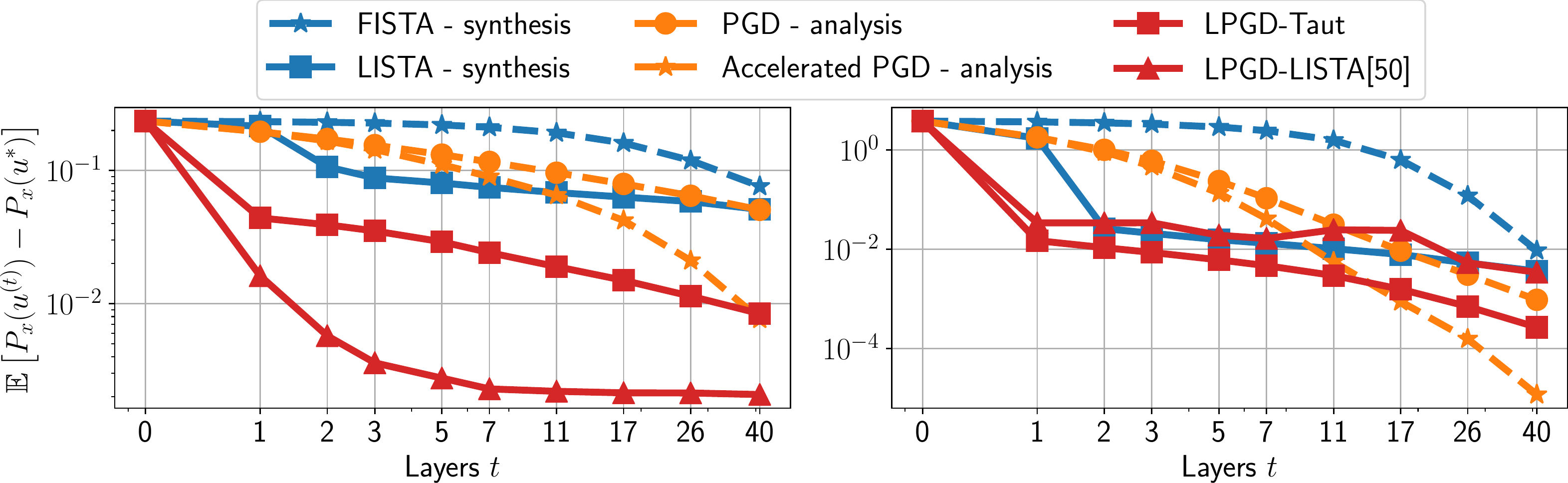}
    \caption{\textbf{Performance comparison} for different regularisation levels (\emph{left}) $\lambda = 0.1$, (\emph{right}) $\lambda = 0.8$. We see that synthesis formulations are outperformed by the analysis counter part. Both our methods are able to accelerate the resolution of \autoref{eq:analysis_formulation_primal}, at least in the first iterations.}
    \label{fig:loss_comparison}
\end{figure}

\autoref{fig:loss_comparison} presents the risk (or expected function value, $P$) of each algorithm as a function of the number of layers or, equivalently, iterations.
For the learned algorithms, the curves in $t$ display the performances of a network with $t$ layer trained specifically.
We observe that all the synthesis formulation algorithms are slower than their analysis counterparts, empirically validating \autoref{prop:lower_bound_AL}.
Moreover, both of the proposed methods accelerate the resolution of \autoref{eq:analysis_formulation_primal} in a low iteration regime.
However, when the regularization parameter is high ($\lambda = 0.8$), we observe that the performance of the LPGD-LISTA tends to plateau.
It is possible that such a high level of sparsity require more than $50$ layers for the inner network (which computes the \proxTV{}).
According to \autoref{sec:inexact_pgd}, the error associated with this proximity step hinders the global convergence, making the loss function decrease slowly.
Increasing the number of inner layers would alleviate this issue, though at the expense of increased computational burden for both training and runtime.
For LPGD-Taut, while the Taut-string algorithm ensures that the recovered support is exact for the proximal step, the overall support can be badly estimated in the first iterations.
This can lead to un-informative gradients as they greatly depend on the support of the solution in this case, and explain the reduced performances of the network in the high sparsity setting.

\paragraph{Inexact \proxTV{}}

With the same data $(x_i)_{i=1}^{n} \in \Rset^{ n \times m}$, we empirically investigate the error of the \proxTV{} $\epsilon_k^{(t)} = F_{u^{(t)}}(z^{(t)}) - F_{u^{(t)}}(z^*)$ and evaluate it for c with different number of layers ($T \in [20, 50]$).
We also investigate the case where the parameter of the nested LISTA in LPGD-LISTA are trained compared to their initialization in untrained version.

\begin{figure}[t]
    \centering
    \includegraphics[width=\columnwidth]{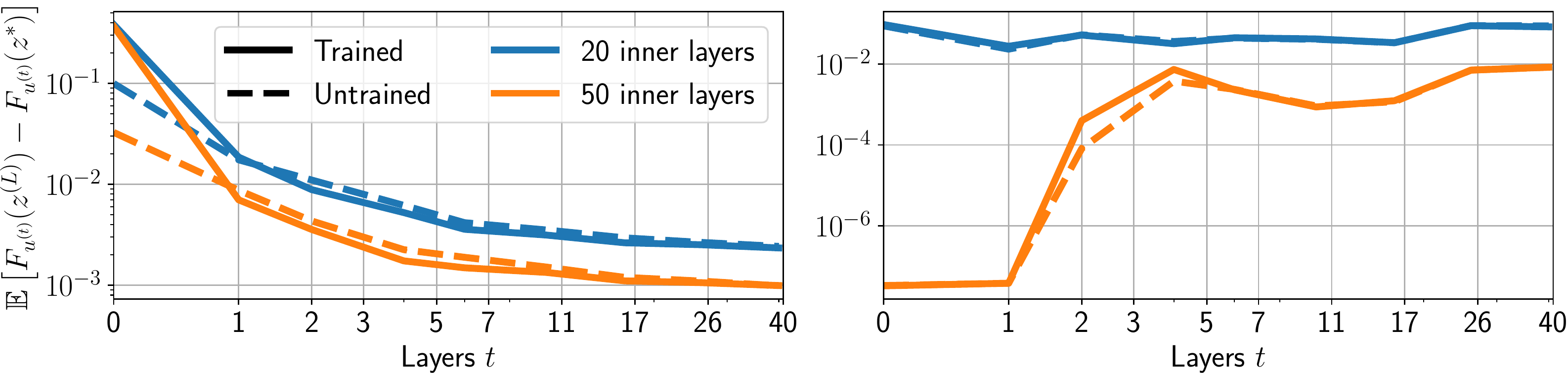}
    \caption{\textbf{Proximal operator error comparison} for different regularisation levels (\emph{left}) $\lambda = 0.1$, (\emph{right}) $\lambda = 0.8$. We see that learn the trained unrolled prox-TV barely improve the performance. More interestingly, in a high sparsity context, after a certain point, the error sharply increase.}
    \label{fig:comparison_prox_tv_loss}
\end{figure}

\autoref{fig:comparison_prox_tv_loss} depicts the error $\epsilon_k$ for each layer. We see that learning the parameters of the unrolled \proxTV{} in LPGD-LISTA barely improves the performance. More interestingly, we observe that in a high sparsity setting the error sharply increases after a certain number of layers. This is likely cause by the high sparsity of the estimates, the small numbers of iterations of the inner network (between 20 and 50) are insufficient to obtain an accurate solution to the proximal operator. This is in accordance with inexact PGD theory which predict that such algorithm has no exact convergence guarantees \citep{Schmidt2011}.

\subsection{fMRI data deconvolution}
\label{sub:expe:fmri}

Functional magnetic resonance imaging (fMRI) is a non-invasive method for recording the brain activity by dynamically measuring  blood oxygenation level-dependent (BOLD) contrast, denoted here $x$.
The latter reflects the local changes in the deoxyhemoglobin concentration in the brain \citet{Ogawa1992} and thus indirectly measures neural activity through the neurovascular coupling.
This coupling is usually modelled as a linear and time-invariant system and characterized by its impulse response, the so-called haemodynamic response function (HRF), denoted here $h$.
Recent developments propose to estimate either the neural activity signal independently \citep{Fikret2013, Cherkaoui2019a} or jointly with the HRF \citep{Cherkaoui2019, Farouj2019}.
Estimating the neural activity signal with a fixed HRF is akin to a deconvolution problem regularized with TV-norm,
\begin{equation}
    \label{eq:analysis_conv}
    \min_{u \in \Rset^{k}} ~ P(u) = \frac 12 \| h * u - x \|_2^2 + \lambda \| u \|_{TV}
\end{equation}
To demonstrate the usefulness of our approach \changed{with real data, where the training set has not the exact same distribution than the testing set}, we compare the LPGD-Taut to Accelerated PGD for the analysis formulation on this deconvolution problem. We choose two subjects from the UK Bio Bank (UKBB) dataset \citep{ukbb2015}, perform the usual fMRI processing and reduce the dimension of the problem to retain only $8000$ time-series of $250$ time-frames, corresponding to a record of 3 minute 03 seconds. The full preprocessing pipeline is described in \autoref{sub:fmri_prepro}. We train the LPGD taut-string network solver on the first subject and \autoref{fig:bold_deconv_loss_comp} reports the performance of the two algorithms on the second subject for $\lambda = 0.1$.
\changed{The performance is reported relatively to the number of iteration as the computational complexity of each iteration or layer for both methods is equivalent.}
It is clear that LPGD-Taut converges faster than the Accelerated PGD even on real data.
In particular, acceleration is higher when the regularization parameter $\lambda$ is smaller.
\changed{As mentioned} previously, this acceleration is likely to be caused by the better learning capacity of the network in a low sparsity context.
The same experiment is repeated for $\lambda = 0.8$ in \autoref{fig:bold_deconv_loss_comp_0_8}.
\begin{figure}[hbtp!]
    \begin{minipage}{.5\textwidth}
        \includegraphics[width=\textwidth]{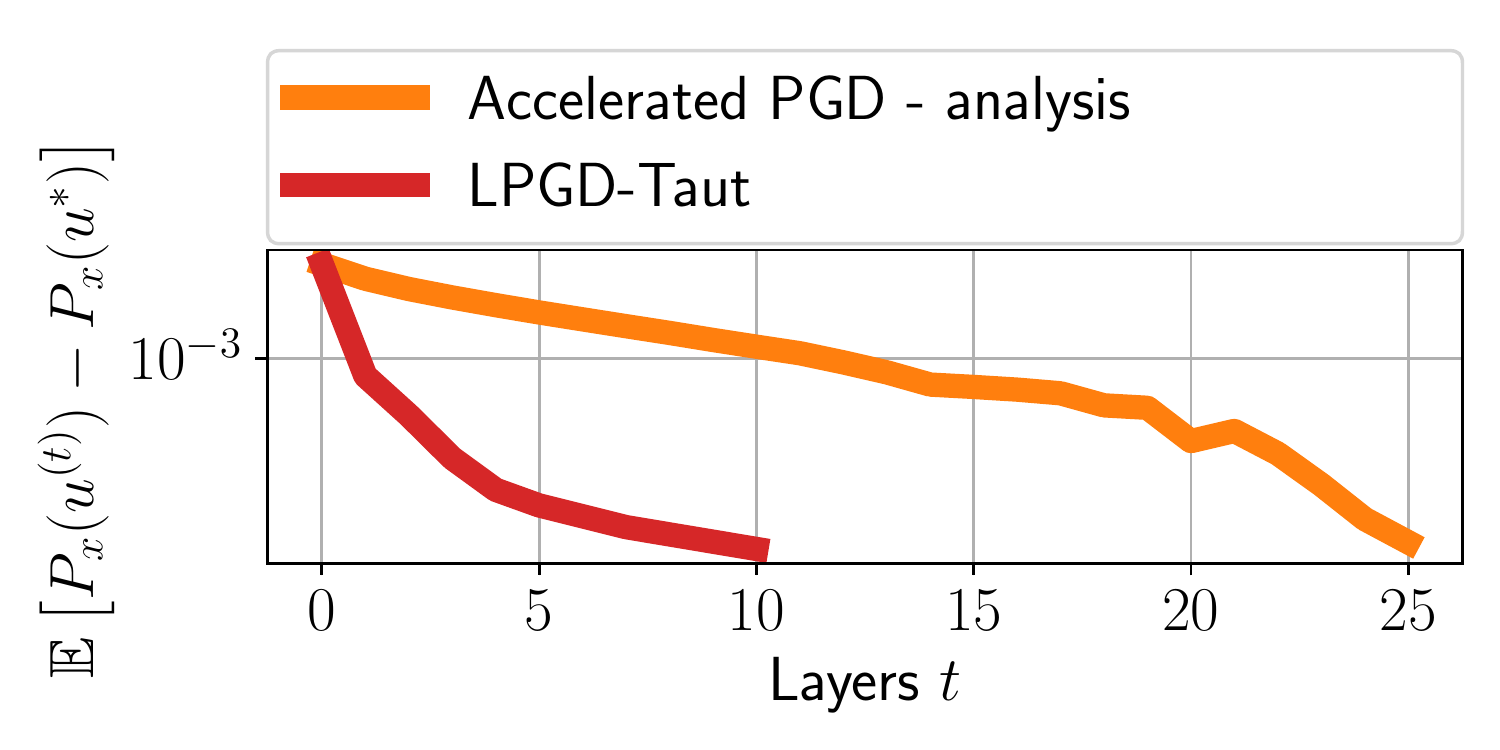}
    \end{minipage}%
    \begin{minipage}{.45\textwidth}
    	\caption{\textbf{Performance comparison} ($\lambda = 0.1$) between our analytic prox-TV derivative method and the PGD in the analysis formulation for the HRF deconvolution problem with fMRI data. Our proposed method outperform the FISTA algorithm in the analysis formulation.}
      \label{fig:bold_deconv_loss_comp}
    \end{minipage}%
\end{figure}
\vskip-.5em
\section{Conclusion}
\label{sec:ccl}

This paper studies the optimization of TV-regularized problems via learned PGD.
We demonstrated, both analytically and numerically, that it is better to address these problems in their original analysis formulation rather than resort to the simpler (alas slower) synthesis version.
We then proposed two different algorithms that allow for the efficient computation and derivation of the required \proxTV{}, exactly or approximately. Our experiments on synthetic and real data demonstrate that our learned networks for \proxTV{} provide a significant advantage in convergence speed.

Finally, we believe that the principles presented in this paper could be generalized and deployed in other optimization problems, involving not just the TV-norm but more general analysis-type priors. In particular, this paper only apply for 1D TV problems because the equivalence between Lasso and TV is not exact in higher dimension. In this case, we believe exploiting a dual formulation \citep{Chambolle2004} for the problem could allow us to derive similar learnable algorithms.

\section*{Broader Impact}

This work attempts to shed some understanding into empirical phenomena in signal processing -- in our case, piecewise constant approximations. As such, it is our hope that this work encourages fellow researchers to invest in the study and development of principled machine learning tools. Besides these, we do not foresee any other immediate societal consequences.



\section*{Acknowledgement}

\changed{
    We gratefully acknowledge discussions with Pierre Ablin, whose suggestions helped us completing some parts of the proofs.
    H. Cherkaoui is supported by a CEA PhD scholarship. J. Sulam is partially supported by NSF Grant 2007649.
}
\bibliographystyle{abbrvnat}
\bibliography{learn_tv}

\begin{thebibliography}{42}
\providecommand{\natexlab}[1]{#1}
\providecommand{\url}[1]{\texttt{#1}}
\expandafter\ifx\csname urlstyle\endcsname\relax
  \providecommand{\doi}[1]{doi: #1}\else
  \providecommand{\doi}{doi: \begingroup \urlstyle{rm}\Url}\fi

\bibitem[Ablin et~al.(2019)Ablin, Moreau, Massias, and Gramfort]{Ablin2019}
P.~Ablin, T.~Moreau, M.~Massias, and A.~Gramfort.
\newblock Learning step sizes for unfolded sparse coding.
\newblock In \emph{Advances in {{Neural Information Processing Systems}}
  ({{NeurIPS}})}, pages 13100--13110, {Vancouver, BC, Canada}, 2019.

\bibitem[{Alfaro-Almagro} et~al.(2018){Alfaro-Almagro}, Jenkinson, Bangerter,
  Andersson, Griffanti, Douaud, Sotiropoulos, Jbabdi, {Hernandez-Fernandez},
  Vidaurre, Webster, McCarthy, Rorden, Daducci, Alexander, Zhang, Dragonu,
  Matthews, Miller, and Smith]{Alfaro-Almagro2018}
F.~{Alfaro-Almagro}, M.~Jenkinson, N.~K. Bangerter, J.~L.~R. Andersson,
  L.~Griffanti, G.~Douaud, S.~N. Sotiropoulos, S.~Jbabdi,
  M.~{Hernandez-Fernandez}, D.~Vidaurre, M.~Webster, P.~McCarthy, C.~Rorden,
  A.~Daducci, D.~C. Alexander, H.~Zhang, I.~Dragonu, P.~M. Matthews, K.~L.
  Miller, and S.~M. Smith.
\newblock Image {{Processing}} and {{Quality Control}} for the first 10,000
  {{Brain Imaging Datasets}} from {{UK Biobank}}.
\newblock \emph{NeuroImage}, 166:\penalty0 400--424, 2018.

\bibitem[Barbero and Sra(2018)]{Barbero2018}
{\`A}.~Barbero and S.~Sra.
\newblock Modular proximal optimization for multidimensional total-variation
  regularization.
\newblock \emph{The Journal of Machine Learning Research}, 19\penalty0
  (1):\penalty0 2232--2313, Jan. 2018.

\bibitem[Beck and Teboulle(2009)]{Beck2009}
A.~Beck and M.~Teboulle.
\newblock A {{Fast Iterative Shrinkage}}-{{Thresholding Algorithm}} for
  {{Linear Inverse Problems}}.
\newblock \emph{SIAM Journal on Imaging Sciences}, 2\penalty0 (1):\penalty0
  183--202, 2009.

\bibitem[Bertrand et~al.(2020)Bertrand, Klopfenstein, Blondel, Vaiter,
  Gramfort, and Salmon]{Bertrand2020}
Q.~Bertrand, Q.~Klopfenstein, M.~Blondel, S.~Vaiter, A.~Gramfort, and
  J.~Salmon.
\newblock Implicit differentiation of {{Lasso}}-type models for hyperparameter
  optimization.
\newblock In \emph{International {{Conference}} on {{Machine Learning}}
  ({{ICML}})}, volume 2002.08943, pages 3199--3210, {online}, Apr. 2020.

\bibitem[Borgerding et~al.(2017)Borgerding, Schniter, and
  Rangan]{Borgerding2017}
M.~Borgerding, P.~Schniter, and S.~Rangan.
\newblock {{AMP}}-{{Inspired Deep Networks}} for {{Sparse Linear Inverse
  Problems}}.
\newblock \emph{IEEE Transactions on Signal Processing}, 65\penalty0
  (16):\penalty0 4293--4308, 2017.

\bibitem[Boyd et~al.(2011)Boyd, Parikh, Chu, Peleato, and Eckstein]{Boyd2011}
S.~Boyd, N.~Parikh, E.~Chu, B.~Peleato, and J.~Eckstein.
\newblock Distributed {{Optimization}} and {{Statistical Learning}} via the
  {{Alternating Direction Method}} of {{Multipliers}}.
\newblock \emph{Foundations and Trends in Machine Learning}, 3\penalty0
  (1):\penalty0 1--122, 2011.

\bibitem[Chambolle(2004)]{Chambolle2004}
A.~Chambolle.
\newblock An {{Algorithm}} for {{Total Variation Minimization}} and
  {{Applications}}.
\newblock \emph{Journal of Mathematical Imaging and Vision}, 20\penalty0
  (1/2):\penalty0 89--97, Jan. 2004.

\bibitem[Chambolle and Pock(2011)]{Chambolle2011}
A.~Chambolle and T.~Pock.
\newblock A {{First}}-{{Order Primal}}-{{Dual Algorithm}} for {{Convex
  Problems}} with {{Applications}} to {{Imaging}}.
\newblock \emph{Journal of Mathematical Imaging and Vision}, 40\penalty0
  (1):\penalty0 120--145, May 2011.

\bibitem[Chebyshev(1853)]{Chebyshev1853}
P.~L. Chebyshev.
\newblock \emph{Th\'eorie Des M\'ecanismes Connus Sous Le Nom de
  Parall\'elogrammes}.
\newblock {Imprimerie de l'Acad\'emie imp\'eriale des sciences}, 1853.

\bibitem[Cherkaoui et~al.(2019{\natexlab{a}})Cherkaoui, Moreau, Halimi, and
  Ciuciu]{Cherkaoui2019}
H.~Cherkaoui, T.~Moreau, A.~Halimi, and P.~Ciuciu.
\newblock Sparsity-based {{Semi}}-{{Blind Deconvolution}} of {{Neural
  Activation Signal}} in {{fMRI}}.
\newblock In \emph{{{IEEE International Conference}} on {{Acoustics}},
  {{Speech}} and {{Signal Processing}} ({{ICASSP}})}, {Brighton, UK},
  2019{\natexlab{a}}.

\bibitem[Cherkaoui et~al.(2019{\natexlab{b}})Cherkaoui, Moreau, Halimi, and
  Ciuciu]{Cherkaoui2019a}
H.~Cherkaoui, T.~Moreau, A.~Halimi, and P.~Ciuciu.
\newblock {{fMRI BOLD}} signal decomposition using a multivariate low-rank
  model.
\newblock In \emph{European {{Signal Processing Conference}} ({{EUSIPCO}})},
  {Coru\~na, Spain}, 2019{\natexlab{b}}.

\bibitem[Combettes and Bauschke(2011)]{Combettes2011}
P.~L. Combettes and H.~H. Bauschke.
\newblock \emph{Convex {{Analysis}} and {{Monotone Operator Theory}} in
  {{Hilbert Spaces}}}.
\newblock {Springer}, 2011.

\bibitem[Condat(2013{\natexlab{a}})]{Condat2013}
L.~Condat.
\newblock A {{Direct Algorithm}} for {{1D Total Variation Denoising}}.
\newblock \emph{IEEE Signal Processing Letters}, 20\penalty0 (11):\penalty0
  1054--1057, 2013{\natexlab{a}}.

\bibitem[Condat(2013{\natexlab{b}})]{Condat2013a}
L.~Condat.
\newblock A {{Primal}}\textendash{{Dual Splitting Method}} for {{Convex
  Optimization Involving Lipschitzian}}, {{Proximable}} and {{Linear Composite
  Terms}}.
\newblock \emph{Journal of Optimization Theory and Applications}, 158\penalty0
  (2):\penalty0 460--479, Aug. 2013{\natexlab{b}}.

\bibitem[Darbon and Sigelle(2006)]{Darbon2006}
J.~Darbon and M.~Sigelle.
\newblock Image {{Restoration}} with {{Discrete Constrained Total Variation
  Part I}}: {{Fast}} and {{Exact Optimization}}.
\newblock \emph{Journal of Mathematical Imaging and Vision}, 26\penalty0
  (3):\penalty0 261--276, Dec. 2006.

\bibitem[Davies and Kovac(2001)]{Davies2001}
P.~L. Davies and A.~Kovac.
\newblock Local {{Extremes}}, {{Runs}}, {{Strings}} and {{Multiresolution}}.
\newblock \emph{The Annals of Statistics}, 29\penalty0 (1):\penalty0 1--65,
  Feb. 2001.

\bibitem[Deledalle et~al.(2014)Deledalle, Vaiter, Fadili, and
  Peyr{\'e}]{Deledalle2014}
C.~A. Deledalle, S.~Vaiter, J.~Fadili, and G.~Peyr{\'e}.
\newblock Stein {{Unbiased GrAdient}} estimator of the {{Risk}} ({{SUGAR}}) for
  multiple parameter selection.
\newblock \emph{SIAM Journal on Imaging Sciences}, 7\penalty0 (4):\penalty0
  2448--2487, 2014.

\bibitem[Elad et~al.(2007)Elad, Milanfar, and Rubinstein]{Elad2007}
M.~Elad, P.~Milanfar, and R.~Rubinstein.
\newblock Analysis versus synthesis in signal priors.
\newblock \emph{Inverse Problems}, 23\penalty0 (3):\penalty0 947--968, June
  2007.

\bibitem[Farouj et~al.(2019)Farouj, Karahanoglu, and Ville]{Farouj2019}
Y.~Farouj, F.~I. Karahanoglu, and D.~V.~D. Ville.
\newblock Bold {{Signal Deconvolution Under Uncertain H\AE Modynamics}}: {{A
  Semi}}-{{Blind Approach}}.
\newblock In \emph{{{IEEE}} 16th {{International Symposium}} on {{Biomedical
  Imaging}} ({{ISBI}})}, pages 1792--1796, {Venice, Italy}, Apr. 2019. {IEEE}.

\bibitem[Fikret et~al.(2013)Fikret, {Caballero-gaudes}, Lazeyras, and
  Dimitri]{Fikret2013}
I.~K. Fikret, C.~{Caballero-gaudes}, F.~Lazeyras, and V.~D.~V. Dimitri.
\newblock Total activation: {{fMRI}} deconvolution through spatio-temporal
  regularization.
\newblock \emph{NeuroImage}, 73:\penalty0 121--134, 2013.

\bibitem[Giryes et~al.(2018)Giryes, Eldar, Bronstein, and Sapiro]{Giryes2016}
R.~Giryes, Y.~C. Eldar, A.~M. Bronstein, and G.~Sapiro.
\newblock Tradeoffs between {{Convergence Speed}} and {{Reconstruction
  Accuracy}} in {{Inverse Problems}}.
\newblock \emph{IEEE Transaction on Signal Processing}, 66\penalty0
  (7):\penalty0 1676--1690, 2018.

\bibitem[Gregor and Le~Cun(2010)]{Gregor10}
K.~Gregor and Y.~Le~Cun.
\newblock Learning {{Fast Approximations}} of {{Sparse Coding}}.
\newblock In \emph{International {{Conference}} on {{Machine Learning}}
  ({{ICML}})}, pages 399--406, 2010.

\bibitem[Lecouat et~al.(2020)Lecouat, Ponce, and Mairal]{Lecouat2020}
B.~Lecouat, J.~Ponce, and J.~Mairal.
\newblock Designing and {{Learning Trainable Priors}} with
  {{Non}}-{{Cooperative Games}}.
\newblock In \emph{Advances in {{Neural Information Processing Systems}}
  ({{NeurIPS}})}, {Vancouver, BC, Canada}, June 2020.

\bibitem[Machart et~al.(2012)Machart, Anthoine, and Baldassarre]{Machart2012}
P.~Machart, S.~Anthoine, and L.~Baldassarre.
\newblock Optimal {{Computational Trade}}-{{Off}} of {{Inexact Proximal
  Methods}}.
\newblock \emph{preprint ArXiv}, 1210.5034, 2012.

\bibitem[Monga et~al.(2019)Monga, Li, and Eldar]{Monga2019}
V.~Monga, Y.~Li, and Y.~C. Eldar.
\newblock Algorithm {{Unrolling}}: {{Interpretable}}, {{Efficient Deep
  Learning}} for {{Signal}} and {{Image Processing}}.
\newblock \emph{preprint ArXiv}, 1912.10557, Dec. 2019.

\bibitem[Moreau and Bruna(2017)]{Moreau2017}
T.~Moreau and J.~Bruna.
\newblock Understanding {{Neural Sparse Coding}} with {{Matrix Factorization}}.
\newblock In \emph{International {{Conference}} on {{Learning Representation}}
  ({{ICLR}})}, {Toulon, France}, 2017.

\bibitem[Ogawa et~al.(1992)Ogawa, Tank, Menon, Ellermann, Kim, Merkle, and
  Ugurbil]{Ogawa1992}
S.~Ogawa, D.~W. Tank, R.~Menon, J.~M. Ellermann, S.~G. Kim, H.~Merkle, and
  K.~Ugurbil.
\newblock Intrinsic signal changes accompanying sensory stimulation: Functional
  brain mapping with magnetic resonance imaging.
\newblock \emph{Proceedings of the National Academy of Sciences}, 89\penalty0
  (13):\penalty0 5951--5955, July 1992.

\bibitem[Paszke et~al.(2019)Paszke, Gross, Massa, Lerer, Bradbury, Chanan,
  Killeen, Lin, Gimelshein, Antiga, Desmaison, Kopf, Yang, DeVito, Raison,
  Tejani, Chilamkurthy, Steiner, Fang, Bai, and Chintala]{Paszke2019}
A.~Paszke, S.~Gross, F.~Massa, A.~Lerer, J.~Bradbury, G.~Chanan, T.~Killeen,
  Z.~Lin, N.~Gimelshein, L.~Antiga, A.~Desmaison, A.~Kopf, E.~Yang, Z.~DeVito,
  M.~Raison, A.~Tejani, S.~Chilamkurthy, B.~Steiner, L.~Fang, J.~Bai, and
  S.~Chintala.
\newblock {{PyTorch}}: {{An Imperative Style}}, {{High}}-{{Performance Deep
  Learning Library}}.
\newblock In \emph{Advances in {{Neural Information Processing Systems}}
  ({{NeurIPS}})}, page~12, {Vancouver, BC, Canada}, 2019.

\bibitem[Rockafellar(1976)]{Rockafellar1976}
R.~T. Rockafellar.
\newblock Monotone {{Operators}} and the {{Proximal Point Algorithm}}.
\newblock \emph{SIAM Journal on Control and Optimization}, 14\penalty0
  (5):\penalty0 877--898, 1976.

\bibitem[Rodr{\'i}guez(2013)]{Rodriguez2013}
P.~Rodr{\'i}guez.
\newblock Total {{Variation Regularization Algorithms}} for {{Images
  Corrupted}} with {{Different Noise Models}}: {{A Review}}.
\newblock \emph{Journal of Electrical and Computer Engineering}, 2013:\penalty0
  1--18, 2013.

\bibitem[Rudin et~al.(1992)Rudin, Osher, and Fatemi]{Rudin1992}
L.~I. Rudin, S.~Osher, and E.~Fatemi.
\newblock Nonlinear total variation based noise removal algorithms.
\newblock \emph{Physica D: Nonlinear Phenomena}, 60\penalty0 (1-4):\penalty0
  259--268, Nov. 1992.

\bibitem[Schmidt et~al.(2011)Schmidt, Le~Roux, and Bach]{Schmidt2011}
M.~Schmidt, N.~Le~Roux, and F.~R. Bach.
\newblock Convergence {{Rates}} of {{Inexact Proximal}}-{{Gradient Methods}}
  for {{Convex Optimization}}.
\newblock In \emph{Advances in {{Neural Information Processing Systems}}
  ({{NeurIPS}})}, pages 1458--1466, {Grenada, Spain}, 2011.

\bibitem[Silverstein(1989)]{Silverstein1989}
J.~W. Silverstein.
\newblock On the eigenvectors of large dimensional sample covariance matrices.
\newblock \emph{Journal of Multivariate Analysis}, 30\penalty0 (1):\penalty0
  1--16, July 1989.

\bibitem[Sprechmann et~al.(2012)Sprechmann, Bronstein, and
  Sapiro]{Sprechmann2012}
P.~Sprechmann, A.~M. Bronstein, and G.~Sapiro.
\newblock Learning {{Efficient Structured Sparse Models}}.
\newblock In \emph{International {{Conference}} on {{Machine Learning}}
  ({{ICML}})}, pages 615--622, {Edinburgh, Great Britain}, 2012.

\bibitem[Sprechmann et~al.(2013)Sprechmann, Litman, and Yakar]{Sprechmann2013a}
P.~Sprechmann, R.~Litman, and T.~Yakar.
\newblock Efficient {{Supervised Sparse Analysis}} and {{Synthesis Operators}}.
\newblock In \emph{Advances in {{Neural Information Processing Systems}}
  ({{NeurIPS}})}, pages 908--916, {South Lake Tahoe, United States}, 2013.

\bibitem[Sudlow et~al.(2015)Sudlow, Gallacher, Allen, Beral, Burton, Danesh,
  Downey, Elliott, Green, Landray, Liu, Matthews, Ong, Pell, Silman, Young,
  Sprosen, Peakman, and Collins]{ukbb2015}
C.~Sudlow, J.~Gallacher, N.~Allen, V.~Beral, P.~Burton, J.~Danesh, P.~Downey,
  P.~Elliott, J.~Green, M.~Landray, B.~Liu, P.~Matthews, G.~Ong, J.~Pell,
  A.~Silman, A.~Young, T.~Sprosen, T.~Peakman, and R.~Collins.
\newblock {{UK Biobank}}: {{An Open Access Resource}} for {{Identifying}} the
  {{Causes}} of a {{Wide Range}} of {{Complex Diseases}} of {{Middle}} and
  {{Old Age}}.
\newblock \emph{PLOS Medicine}, 12\penalty0 (3):\penalty0 e1001779, Mar. 2015.

\bibitem[Sulam et~al.(2019)Sulam, Aberdam, Beck, and Elad]{Sulam2019}
J.~Sulam, A.~Aberdam, A.~Beck, and M.~Elad.
\newblock On {{Multi}}-{{Layer Basis Pursuit}}, {{Efficient Algorithms}} and
  {{Convolutional Neural Networks}}.
\newblock \emph{IEEE Transactions on Pattern Analysis and Machine Intelligence
  (PAMI)}, 2019.

\bibitem[Tian et~al.(2011)Tian, Jia, Yuan, Pan, and Jiang]{Tian2011}
Z.~Tian, X.~Jia, K.~Yuan, T.~Pan, and S.~B. Jiang.
\newblock Low-dose {{CT}} reconstruction via edge-preserving total variation
  regularization.
\newblock \emph{Physics in Medicine and Biology}, 56\penalty0 (18):\penalty0
  5949--5967, Sept. 2011.

\bibitem[Tibshirani(1996)]{Tibshirani1996}
R.~Tibshirani.
\newblock Regression {{Shrinkage}} and {{Selection}} via the {{Lasso}}.
\newblock \emph{Journal of the Royal Statistical Society: Series B (statistical
  methodology)}, 58\penalty0 (1):\penalty0 267--288, 1996.

\bibitem[Tibshirani and Taylor(2011)]{Tibshirani2011}
R.~J. Tibshirani and J.~Taylor.
\newblock The solution path of the generalized lasso.
\newblock \emph{The Annals of Statistics}, 39\penalty0 (3):\penalty0
  1335--1371, June 2011.

\bibitem[Xin et~al.(2016)Xin, Wang, Gao, and Wipf]{Xin2016}
B.~Xin, Y.~Wang, W.~Gao, and D.~Wipf.
\newblock Maximal {{Sparsity}} with {{Deep Networks}}?
\newblock In \emph{Advances in {{Neural Information Processing Systems}}
  ({{NeurIPS}})}, pages 4340--4348, 2016.

\end{thebibliography}

\newpage
\appendix

\counterwithin{figure}{section}

\changed{
\section{Network training process strategy}
\label{sec:training_process}

    Here, we give a more detailed description of the training procedure used in \autoref{sec:exp}.

    \paragraph{Optimization algorithm for training}
    In our experiments, all networks are trained using Gradient Descent (GD) with back-tracking line search. The gradients are computed using automatic differentiation in Pytorch~\citep{Paszke2019} for most layers and the weak Jacobian proposed in \autoref{sub:prox_jacobian} for the back-propagation through the \proxTV{}. The learning is stopped once a step-size of $\eta_{limit} = 10^{-20}$ is reached in the back-tracking step. For LPGD-LISTA, the weights of the inner LISTA computing the \proxTV{} are trained jointly with the parameters of the outer unrolled PGD.

    \paragraph{Weight initialization}
    All layers for an unrolled algorithm are initialized using the values of weights in \autoref{eq:init_weights} that ensure the output of the layer with $T$ layers corresponds to the output of $T$ iterations of the original algorithm. To further stabilize the training, we use a layer-wise approach. When training a network with $T_1 + T_2$ layers after having trained a network with $T_1$ layers, the first $T_1$ layers in the new network are initialized with the weights of the one trained previously, and the remaining layers are initialized using weights value from \autoref{eq:init_weights}. This ensures that the initial value of the loss for the new network is smaller than the one from the shallower one if the unrolled algorithm is monotonous (as it is the case for PGD).

}

\section{Real fMRI data acquisition parameters and preprocessing strategy}
\label{sub:fmri_prepro}

In this section, we complete the description of the resting-state fMRI (rs-fMRI) data used for the experiment of Fig.~\ref{fig:bold_deconv_loss_comp}. For this experiment, we investigate the 6 min long rs-fMRI acquisition (TR=0.735~s) from the UK Bio Bank dataset \citep{ukbb2015}. The following pre-processing steps were applied on the images: motion correction, grand-mean intensity normalisation, high-pass temporal filtering, Echo planar imaging unwarping, Gradient Distortion Correction unwarping and structured artefacts removal by Independant Components Analysis. \changed{More details on the processing pipeline} can found in \citet{Alfaro-Almagro2018}.

On top of this preprocessing, we perform a standard fMRI preprocessing proposed in the python package Nilearn\footnote{\texttt{https://nilearn.github.io}}. This standard pipeline includes to detrend the data, standardize it and filter high and low frequencies to reduce the presence of noise.

\section{Real fMRI data experiment addition results}

Here, we provide an extra experiment for \autoref{sub:expe:fmri} with $\lambda = 0.8\lambda_{\max}$ and recall the previous one with $\lambda = 0.1\lambda_{\max}$ to help performance comparison in different regularization regime.

\begin{figure}[hbtp!]
    \begin{subfigure}[b]{.5\textwidth}
        \includegraphics[width=\textwidth]{bold_deconv_loss_comp_0_1.pdf}
        \caption{$\lambda = 0.1\lambda_{\max}$}
    \end{subfigure}%
    \begin{subfigure}[b]{.5\textwidth}
        \includegraphics[width=\textwidth]{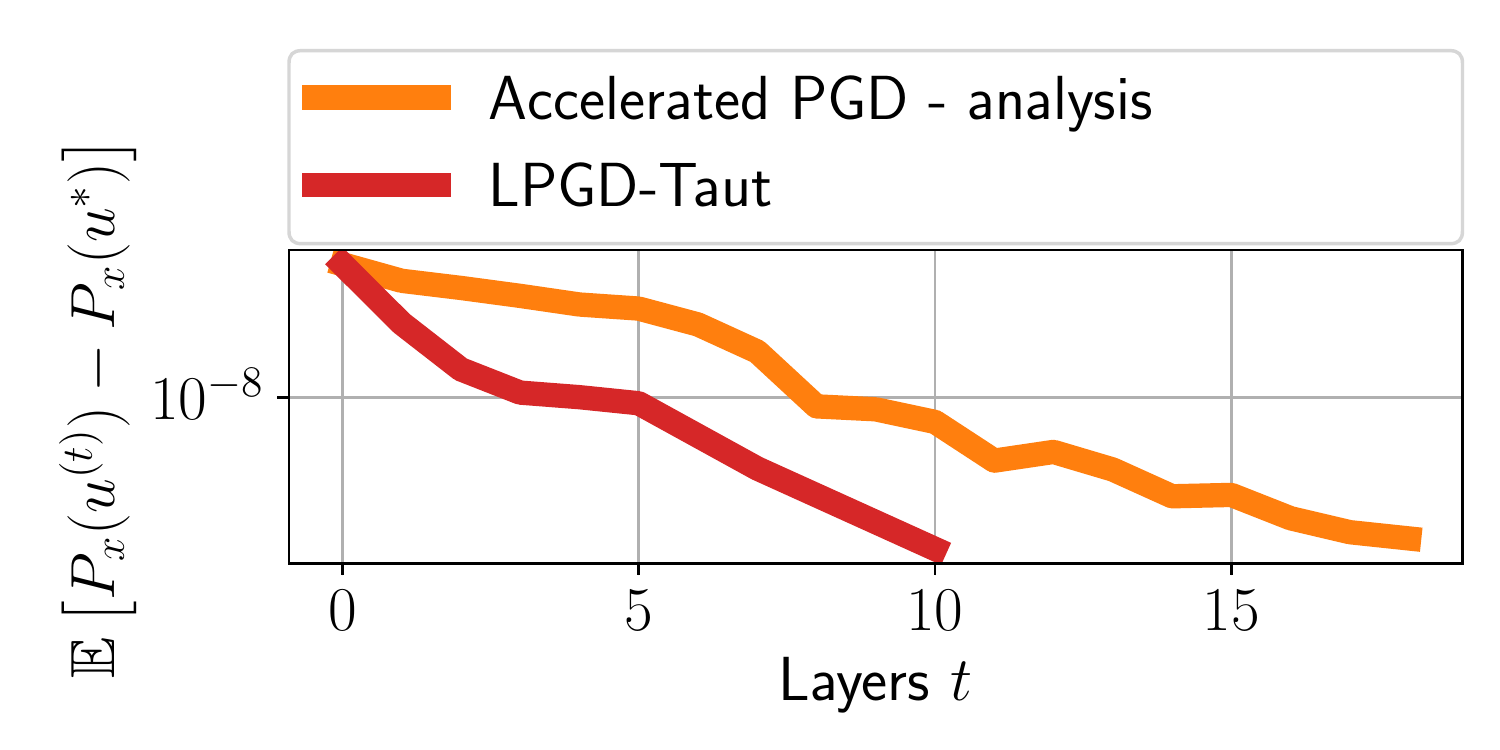}
        \caption{$\lambda = 0.8\lambda_{\max}$}
    \end{subfigure}
    	\caption{\textbf{Performance comparison} between LPGD-Taut and iterative PGD for the analysis formulation for the HRF deconvolution problem with fMRI data. Our proposed method outperform the FISTA algorithm in the analysis formulation. We notice a slight degradation of the acceleration in this high sparsity context.}
      \label{fig:bold_deconv_loss_comp_0_8}
\end{figure}

We can see that the performance drop in the higher sparsity setting compared to the performances with $\lambda = 0.1\lambda_{\max}$ but LPGD-Taut still outperforms iterative algorithms for this task on real data.

\section{Computing \texorpdfstring{$\lambda_{\max}$}{lambda max} for the TV regularized problem}
\label{sub:lambda_max}

The definition of $\lambda_{\max}$ is the smallest value for the regularisation parameter $\lambda$ such that the solution of the $TV$-regularized problem is constant. This corresponds to the definition of $\lambda_{\max}$ in the Lasso, which is the smallest regularisation parameter such that $0$ is solution. We here derive its analytic value which is used to rescale all experiments. This is important to define an equiregularisation set for the training and testing samples, to have a coherent and generalizable training.

\begin{proposition}
The value of $\lambda_{\max}$ for the TV-regularized problem is
$$
    \lambda_{\max} = \|A^\top(Ac\mathbf{1} - x)\|_\infty
$$
where $c = \displaystyle\frac{\sum_{i=1}^p S_i x_i}{\sum_{i=1}^p S_i^2}$ and $S_i = \sum_{j=1}^k A_{i,j}$.
\end{proposition}

\begin{proof}
We first derive the constant solution of the $\ell_2$-regression problem associated to \autoref{eq:tv_analysis}.
For $c\in\Rset$, we consider a constant vector $c \mathbf{1}$.
The best constant solution for the $\ell_2$-regression problem is obtained by solving
\begin{equation}
    \min_{c \in \Rset}f_x(c) = \frac12\|x - cA\mathbf{1}\|_2^2
    \enspace.
\end{equation}
The first order optimality condition in c reads
\begin{equation}
    \nabla f_x(c)
      = \sum_{i=1}^n (\sum_{j=1}^k A_{i,j})(c\sum_{j=1}^k A_{i,j} - x_i)
      = \sum_{i=1}^n S_i(c S_i - x_i)
      = 0
     \enspace ,
\end{equation}
and thus $c = \displaystyle\frac{\sum_{i=1}^p S_i x_i}{\sum_{i=1}^p S_i^2}$.

Then, we look at the conditions on $\lambda$ to ensure that the constant solution $c\mathbf{1}$ is solution of the regularized problem. The first order conditions on the regularized problem reads

\begin{equation}
    0 \in \partial P_x(c\mathbf{1}) = A^\top (Ac\mathbf{1} - x) + \lambda \partial \| D c\mathbf{1} \|_1
\end{equation}

Next, we develop the previous equality:
\begin{equation}
    \forall j\in \{2, \dots k\},\quad A_j^\top (Ac\mathbf{1} - x) \in \lambda \partial (\| D c\mathbf{1} \|_1)_j = [-\lambda, \lambda]
    \quad \text{since} ~ Dc\mathbf{1} = 0
\end{equation}

Thus, the constrains are all satisfied for $\lambda \ge \lambda_{\max}$, with  $\lambda_{max} = \|A^\top(Ac\mathbf{1} - x)\|_\infty$ and as $c$ is solution for the unregularized problem reduced to a constant, $c\mathbf{1}$ is solution of the TV-regularized problem for all $\lambda \ge \lambda_{\max}$.

\end{proof}

\section{Dual formulation}

In this work, we devote our effort in the analysis formulation depicted in \autoref{eq:tv_analysis}. In this section, we propose to investigate the dual formulation corresponding to \autoref{eq:tv_analysis} in order to rationalize our choice to focus on approaches that solve the \proxTV{} with an iterative method.

\paragraph{Dual derivation}

First, we derive the dual of the analysis formulation for the \proxTV{}.

\begin{restatable}{proposition}{DualDerivation}[Dual re-parametrization for the analysis formulation TV problem \autoref{eq:tv_analysis}]
    \label{prop:dual_derivation}

    Considering the primal analysis problem (with operator and variables defined as previously)

    \begin{equation}
        \label{eq:analysis_formulation_primal}
        P_x(u) = \frac{1}{2} \| x - A u \|_2^2 + \lambda \| D u \|_1
    \end{equation}

    Then,
    the dual formulation reads:
    \begin{align}
        \label{eq:explicit_dual_problem}
         & p = - \min_v ~ \frac{1}{2} \| {A^{\dagger}}^\top D^\top v \|_2^2 - v^\top D  A^{\dagger} x \\
         & \st ~ \|v\|_{\infty} \leq \lambda
    \end{align}

\end{restatable}

\begin{proof}

Defining, $f$ and $g$, such as  $f(u) = \frac{1}{2} \| x - A u \|_2^2$ and $g(u) = \lambda \| u \|_1$ and by denoting $p$ the minimum of \autoref{eq:analysis_formulation_primal} \wrt{} $u$, the problem reads:

\begin{equation}
    p = \min_u \quad f(u) + g(Du)
\end{equation}

With the Fenchel-Rockafellar duality theorem, we derive the dual re-parametrization:

\begin{equation}
    \label{eq:dual_problem}
    p = - \min_v ~ f^*(-D^\top v) + g^*(v)
\end{equation}

Note, that in this case we have the equality with $p$ since the problem \autoref{eq:tv_analysis} is $\mu$-strongly convex, with $\mu = \frac 12$.

We have $g^*(v) = -\min_u ~ g(u) - v^\top u$. With a
component-wise minimization, we obtain $g^*(v)_i = \delta_{|v_i| \leq \lambda}$ with $\delta$ being the convex indicator. Thus, we deduce that
$g^*(v) = \delta_{\|v\|_{\infty} \leq \lambda}$.

Then, we have $f^*(v) = -\min_u ~ f(u) - v^\top u$. By
cancelling the gradient we obtain: $f^*(v) = \frac{1}{2} \| {A^\dagger}^\top v \|_2^2 + v^\top A^{\dagger} x$

This allows use to conclude the demonstration. Note that, if we set $A = \Id$, we obtain the same problem as \citep{Barbero2018, Chambolle2004}.

\end{proof}

\paragraph{Performance comparison}

We propose to compare the performance of different iterative solvers to assess their performance.

We generate $n=1000$ times series to compare the performance between the different algorithms. We set the length of the source signals $(u_i)_{i=1}^{n} \in \Rset^{ n \times k}$ to $k=40$ with a support of $|S|=4$ non-zero coefficients.
We generate $A \in \Rset^{m \times k}$ as a Gaussian matrix with $m=40$, obtaining then $(u_i)_{i=1}^{n} \in \Rset^{ n \times p}$. Moreover, we add Gaussian noise to measurements $x_i = Au_i$ with a signal to noise ratio (SNR) of $1.0$.

We select the PGD and its accelerated version with the synthesis primal formulation \autoref{eq:tv_synthesis} (``Synthesis primal A/PGD``), the PGD and its accelerated version with the analysis primal formulation (``Analysis primal A/PGD``). We consider also the PGD and its accelerated version \citep{Chambolle2004}, for the analysis dual formulation (``Analysis dual A/PGD``) and finally we add the primal/dual algorithm \citep{Condat2013a} for the analysis primal formulation (``Analysis primal dual GD``).

\begin{figure}[hbtp!]
    \begin{subfigure}[b]{.5\textwidth}
        \includegraphics[width=\textwidth]{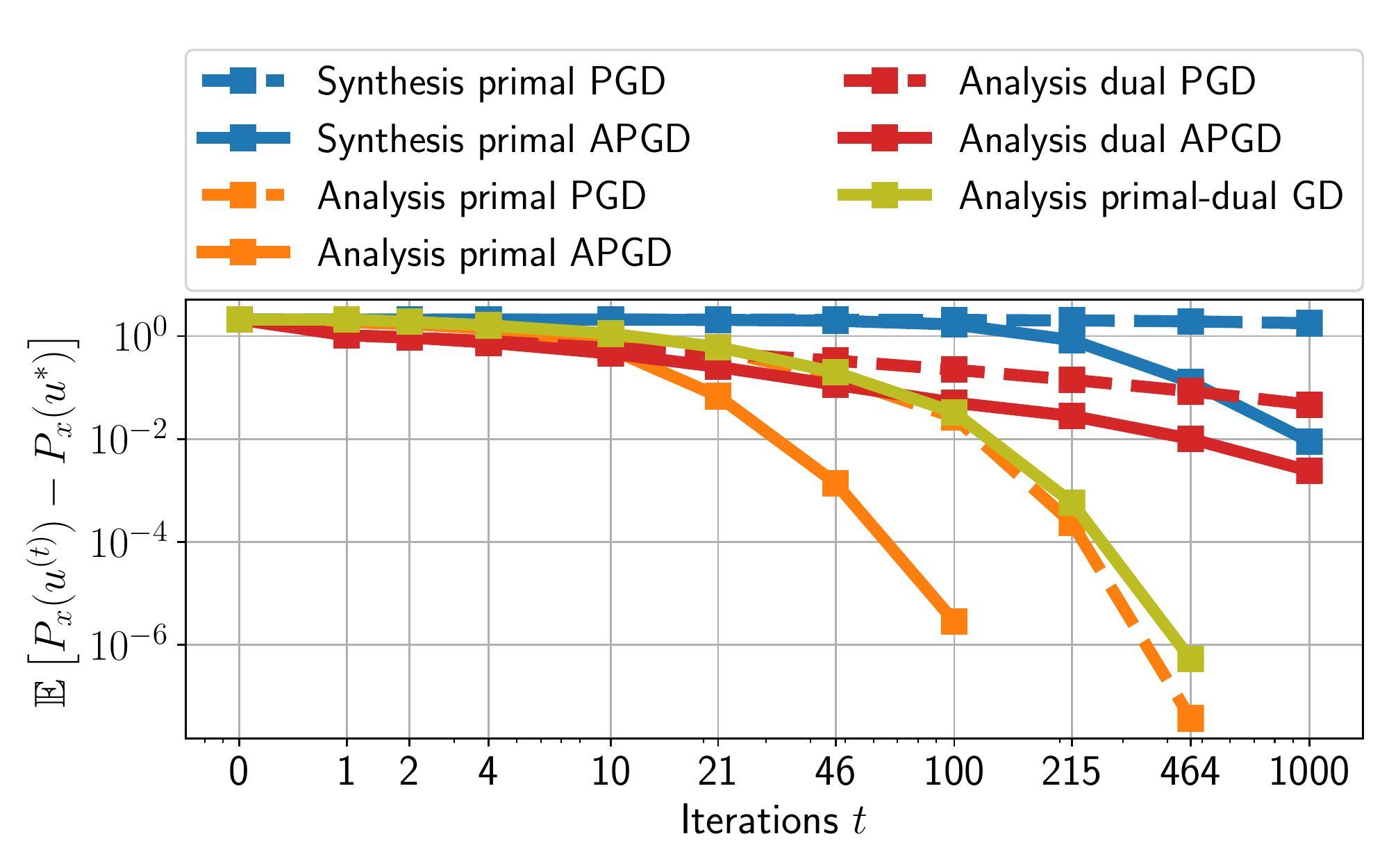}
        \caption{$\lambda = 0.1\lambda_{\max}$}
    \end{subfigure}%
    \begin{subfigure}[b]{.5\textwidth}
        \includegraphics[width=\textwidth]{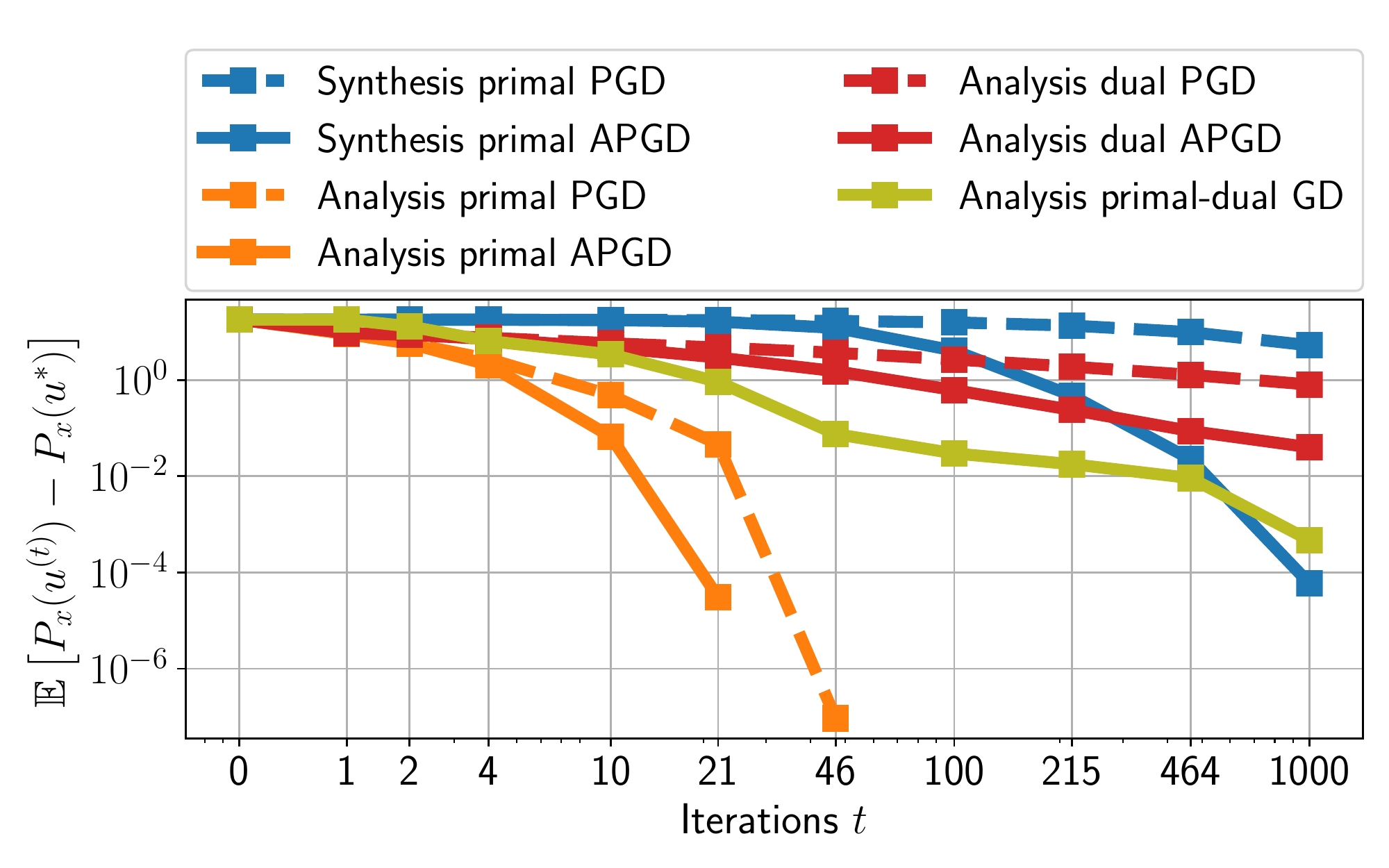}
        \caption{$\lambda = 0.8\lambda_{\max}$}
    \end{subfigure}
    	\caption{\textbf{Performance comparison} between the iterative solver for the synthesis and analysis formulation with the corresponding primal, dual or primal-dual re-parametrization. We notice that the primal analysis formulation provides the best performance in term of iterations. We also observe that the higher the regularization parameter, the faster the performance for each algorithm.}
      \label{fig:iter_loss_comp}
\end{figure}

\autoref{fig:iter_loss_comp} a proposes performance comparison for an exhaustive selection of the algorithm used to solve \autoref{eq:tv_analysis}. We see that the analysis primal formulation proposes the best performance for each regularization parameter. We notice that the \citet{Condat2013a} provides good performance too. Finally, the synthesis primal formulation along with the analysis dual formulation provides the slowest performance.
Those results reinforces our choice to focus on the PGD of the analysis primal formulation.

\section{Proof for \autoref{sec:iterative_algorithm}}

\subsection{Convergence rate of PGD for the synthesis
            formulation \texorpdfstring{\autoref{eq:cvg_rate_synthesis}}{}}
\label{sub:proof:cvg_rate_synthesis}

\begin{proof}
    The convergence rate of ISTA for the synthesis formulation reads
    \begin{equation}
        S(z^{(t)}) - S(z^*) \le \frac{\widetilde \rho}{2t}\|z^{(0)} - z^*\|_2^2
        \enspace .
    \end{equation}
    We use $S(z^{(t)}) = P(Lz^{(t)}) = P(u^{(t)})$ to get the correct left-hand side term.
    For the right hand side, we use $z^{(0)} = \widetilde D u^{(0)}$,  and $z^* = \widetilde D u^*$, which gives $\|z^{(0)} - z^*\|_2^2 = \|\widetilde D(u^{(0)} - u^*)\|_2^2 \le 4 \|u^{(0)} - u^*\|_2^2$.
    The last majoration comes from the fact that $\|\widetilde D\|_2^2 \leq 4$, as shown per \autoref{lemma:spectral_radius_L}.
    This yields
    \begin{equation}
        P(u^{(t)}) - P(u^*) \le \frac{2\widetilde \rho}{t}\|u^{(0)} - u^*\|_2^2
        \enspace .
    \end{equation}
\end{proof}

\subsection{Computing the spectrum of \texorpdfstring{$L$}{}}
\label{sub:proof:cvg_synthesis}

\begin{restatable}{lemma}{spectraL}[Singular values of $L$]
    \label{lemma:spectral_radius_L}
    The singular values of $L \in \Rset^{k\times k}$ are given by
    \[
        \sigma_{l} =
            \frac{1}{2\cos(\frac{\pi l}{2k+1})},
            \qquad \forall l \in \{1,\dots, k\}
            \enspace .
    \]
    Thus, $\|L\|_2 = \frac{2k + 1}{\pi} + o(1)$.
\end{restatable}

\begin{proof}
    As $L$ is invertible, so is $L^\top L$.
    To compute the singular values $\sigma_l$ of $L$, we will compute the eigenvalues $\mu_l$ of $(L^\top L)^{-1}$ and use the relation
    \begin{equation}
        \label{eq:proof:singular_eigen}
        \sigma_l = \frac{1}{\sqrt{\mu_l}}
    \end{equation}
    With simple computations, we obtain
    \begin{equation}
        M_k = (L^\top L)^{-1}
            = L^{-1}(L^\top)^{-1}
            = \widetilde D\widetilde D^\top
            = \begin{bmatrix}
            1 & -1 & 0 & \dots\\
            -1 & 2 & -1 & 0 & \dots\\
             &  \ddots & \ddots & \ddots\\
             & 0 &-1& 2 & -1\\
             &  &  & -1 & 2\\
            \end{bmatrix}
    \end{equation}
    This matrix is tri-diagonal with a quasi-toepliz structure. Its characteristic polynomial $P_k(\mu)$ is given by:
    \begin{align}
       P_k(\mu) = |\mu \Id - M_k|
       & = \begin{vmatrix}
            \mu - 1 & 1 & 0 & \dots\\
            1 & \mu - 2 & 1 & 0 & \dots\\
             &  \ddots & \ddots & \ddots\\
             & 0 &1& \mu - 2 & 1\\
             &  &  0 & 1 & \mu - 2\\
        \end{vmatrix}\\
        & = (\mu - 1) Q_{k-1}(\mu) - Q_{k-2}(\mu)
        \label{eq:proof:dev_Pk}
    \end{align}
    where \autoref{eq:proof:dev_Pk} is obtained by developing the determinant relatively to the first line and $Q_k(\mu)$ is the characteristic polynomial of matrix $\widetilde M_k$ equal to $M_k$ except for the the top left coefficient which is replaced by $2$
    \begin{equation}
        \widetilde M_k = \begin{bmatrix}
            2 & -1 & 0 & \dots\\
            -1 & 2 & -1 & 0 & \dots\\
             &  \ddots & \ddots & \ddots\\
             & 0 &-1& 2 & -1\\
             &  & 0 & -1 & 2\\
            \end{bmatrix}
    \end{equation}
    Using the same development as \autoref{eq:proof:dev_Pk}, one can show that $Q_k$ verifies the recurrence relationship
    \begin{equation}
        Q_k(\mu) = (\mu - 2)Q_{k-1}(\mu) - Q_{k-2}(\mu);
            \qquad Q_{1}(\mu) = 2 - \mu,
            \quad Q_0(\mu) = 1
        \enspace .
    \end{equation}
    Using this with \autoref{eq:proof:dev_Pk} yields
    \begin{equation}
        \label{eq:proof:PQ_relation}
        P_k(\mu) = Q_k(\mu) + Q_{k-1}(\mu)
    \end{equation}
    With the change of variable $\nu = \frac{\mu - 2}{2}$ and denoting $U_k(\nu) = Q_k(2 + 2\nu)$, the recursion becomes
    \begin{equation}
        U_k(\nu) = 2\nu U_{k-1}(\nu) - U_{k-2}(\nu);
        \qquad U_{1}(\nu) = 2\nu,
        \quad U_0(\mu) = 1
        \enspace .
    \end{equation}
    This recursion defines the Chebyshev polynomials of the second kind \citep{Chebyshev1853} which verifies the following relation
    \begin{equation}
        \forall \theta \in [0, 2\pi], \quad
        U_k(\cos(\theta))\sin(\theta) =
          \sin((k+1)\theta)
          \enspace .
    \end{equation}
    Translating this relationship to $Q_k$ gives
    \begin{equation}
        \forall \theta \in [0, 2\pi], \quad
        Q_k(2 + 2\cos(\theta))\sin(\theta) =
          \sin((k+1)\theta)
          \enspace .
    \end{equation}
    Using this in \autoref{eq:proof:PQ_relation} shows that for $\theta \in [0, 2\pi[$ $P_k$ verify
    \begin{equation}
        P_k(2 + 2\cos(\theta)\sin(\theta)
             = \sin((k+1)\theta) + \sin(k\theta)
                \enspace .
    \end{equation}
    The equation
    \begin{equation}
        \sin((k+1)\theta) + \sin(k\theta) = 0
        \enspace ,
    \end{equation}
    has $l$ solution in $[0, 2\pi[$ that are given by $\theta_l = \frac{2\pi l}{2k+1}$ for $l \in \{1, \dots n\}$.
    As for all $l$, $\sin(\theta_l) \neq 0$, the values $\mu_l = 2 + 2\cos(\theta_l) = 4\cos^2(\frac{\pi l}{2k+1})$ are the roots of $P_k$ and therefor the eigenvalues of $M_k$.
    Using \autoref{eq:proof:singular_eigen} yields the expected value for $\sigma_l$.

    The singular value of $L$ is thus obtain for $l = k$ and we get
    \begin{align}
        \|L\|_2 = \sigma_k
            & = \frac{1}{2\cos(\frac{\pi k }{2k+1})}
             = \frac{1}{2\cos(\frac{\pi}{2}(1 - \frac{1}{2k+1}))}
            \enspace ,\\
            & = \frac{1}{2\sin(\frac{\pi}{2}\frac{1}{2k+1})}
             = \frac{2k + 1}{\pi} + o(1)
            \enspace .
    \end{align}
    Where the last approximation comes from $\frac{1}{sin(x)} = 1/x + o(1)$ when $x$ is close to $0$.
\end{proof}

\subsection{Proof for \autoref{prop:lower_bound_AL}}
\label{sub:proof:prop:lower_bound_AL}

\speedLowerBound*

\begin{proof}
    Finding the norm of $AL$ can be written as
    \begin{align}
        \label{eq:proof:eigen_def}
       \|AL\|_2^2 = \max_{x\in\Rset^k} x L^\top A^\top A L x;
            \quad \st \|x\|_2 = 1
    \end{align}
    From \autoref{lemma:spectral_radius_L}, we can write $L= W^\top\Sigma V$ with $V$, $W$ two unitary matrices and $\Sigma$ a diagonal matrix with $\Sigma_{l,l} = \sigma_l$ for all $l\in \{1,..,k\}$.

    First, we consider the case where $A^\top A$ is a rank one matrix with $A^\top A = \|A\|_2^2 u_1u_1^\top$, with vector $u_1$ uniformly sampled from the $\ell_2$-ball and fixed $\|A\|_2$.
    In this case, as $W$ is unitary, $w_1 = Wu_1$ is also a vector uniformly sampled from the sphere.
    Also as $V$ is unitary, it is possible to re-parametrize \autoref{eq:proof:eigen_def} by $y = Vx$ such that
    \begin{equation}
        \max_{y\in\Rset^k} \|A\|_2^2y^\top\Sigma w_1w_1^\top\Sigma y;
            \quad \st \|y\|_2 = 1
    \end{equation}
    This problem can be maximized by taking $y = \frac{\Sigma u_1}{\|\Sigma u_1\|_2}$, which gives
    \begin{equation}
        \|AL\|_2^2 = \|A\|_2^2\|\Sigma w_1\|_2^2
    \end{equation}
    Then, we compute the expectation of $\|\Sigma w_1\|_2^2$ with respect with $w_1$, a random vector sampled in the $\ell_2$ unit ball,
    \begin{align}
        \mathbb E_{w_1}[\|\Sigma w_1\|_2^2]
            = &  \sum_{l=1}^k\sigma_l^2\mathbb E[u_{1,i}^2]
            =  \sum_{l=1}^k \frac{1}{4\cos^2{\frac{\pi l}{2k + 1}}}\frac{1}{k}
            =  \frac{1}{2\pi}\sum_{l=1}^k \frac{\pi}{2k}\frac{1}{\cos^2{\frac{\pi l}{2k + 1}}}
            \enspace .
    \end{align}
    Here, we made use of the fact that for a random vector $u_1$ on the sphere in dimension $k$, $\mathbb E[u_{1, i}] = \frac{1}{k}$
    In the last part of the equation, we recognize a Riemann sum for the interval $[0, \frac{\pi}{2}[$.
    However, $x\mapsto \frac{1}{\cos^2(x)}$ is not integrable on this interval. As the function is positive and monotone, we can still use the integral to highlight the asymptotic behavior of the series.
    For $k$ large enough, we consider the integral
    \begin{equation}
        \int_0^{\frac{\pi}{2} - \frac{\pi}{2k+1}} \frac{1}{\cos^2({x})} dx
        = \left[\frac{sin(x)}{\cos(x)}\right]_{0}^{\frac{\pi}{2} - \frac{\pi}{2k+1}}
        = \frac{\cos{ \frac{\pi}{2k+1}}}{\sin{ \frac{\pi}{2k+1}}} =  \frac{2k+1}{\pi} + o(1)
    \end{equation}
    Thus, for $k$ large enough, we obtain
    \begin{equation}
        \mathbb E_{w_1}\left[\|\Sigma w_1\|_2^2\right] =  \frac{1}{2\pi}\left(\frac{2k+1}{\pi} + o(1)\right)
    \end{equation}
    Thus, we get
    \begin{equation}
        \mathbb E\left[\frac{\|AL\|_2^2}{\|A\|_2^2}\right]
            = \left(\frac{k+\frac12}{\pi^2} + o(1)\right)
    \end{equation}
    This concludes the case where $A^\top A$ is of rank-1 with uniformly distributed eigenvector.

    In the case where $A^\top A$ is larger rank, it is lower bounded by $\|A\|_2^2u_1u_1^\top$ where $u_1$ is its eigenvector associated to its largest eigenvalue, since it is \psd{}.
    Since $A^\top A$ is a Whishart matrix, its eigenvectors are uniformly distributed on the sphere \citep{Silverstein1989}.
     We can thus use the same lower bound  as previously for the whole matrix.

\end{proof}

\section{Proof for \autoref{sec:prox_backprop}}
\label{sub:proof:prox_backprop}

\subsection{Proof for \autoref{prop:jacobian_prox_tv}}
\label{sub:proof:prop:jacobian_prox_tv}

\jacobianProxTV*

First, we recall \autoref{lemma:st_weak_diff} to weakly derive the soft-thresholding. \\

\begin{lemma}[Weak derivative of the soft-thresholding; \citealt{Deledalle2014}]
\label{lemma:st_weak_diff}
The soft-thresholding operator $\text{ST}: \Rset \times \Rset_+ \mapsto \Rset$ defined by $\text{ST}(t, \tau) = \sign(t)(|t| - \tau)_+$ is weakly differentiable with weak derivatives
\[
        \frac{\partial \ST}{\partial t}(t, \tau) =
        \ind_{\{|t| > \tau\}} \enspace ,
        \qquad\text{ and }\qquad
        \frac{\partial \ST}{\partial \tau}(t, \tau) =
        - \sign(t) \cdot \ind_{\{|t| > \tau\}} \enspace,
\]
where
\[
   \ind_{\{|t| > \tau\}} =
   \begin{cases}
       1, & \text{ if } |t| > \tau, \\
       0, & \text{ otherwise}.
       \end{cases}
   \]

\end{lemma}

A very important remark here is to notice that if one denote $z = \ST(t, \tau)$, one can rewrite these weak derivatives as
\begin{equation}
        \frac{\partial \ST}{\partial t}(t, \tau) =
        \ind_{\{|z| > 0\}} \enspace ,
        \qquad\text{ and }\qquad
        \frac{\partial \ST}{\partial \tau}(t, \tau) =
        - \sign(z) \cdot \ind_{\{|z| > 0\}} \enspace.
\end{equation}
Indeed, when $|t| > \tau$, $|z| = |t| - \tau > 0$ and conversely, $|z| = 0$ when $|t| < \tau$. Moreover, when $|t| > \tau$, we have $\sign(t) = \sign(z)$ and thus the two expressions for $\frac{\partial\ST}{\partial \tau}$ match.

Using this \autoref{lemma:st_weak_diff}, we now give the proof of \autoref{prop:jacobian_prox_tv}.

\begin{proof}
    The proof is inspired from the proof from \citet[Proposition 1]{Bertrand2020}.
    We denote $u(x, \mu) = \prox_{\mu\|\cdot\|_{TV}}(x)$, hence $u(x, \mu)$ is defined by
    \begin{equation}
        \label{eq:proof:tv_loss}
        u(x, \mu) = \argmin_{\hat u} \frac12 \|x - \hat u\|_2^2
                    + \mu\|\hat u\|_{TV}
    \end{equation}
    Equivalently, as we have seen previously in \autoref{eq:tv_synthesis}, using the change of variable $\hat u = L \hat z$ and minimizing over $\hat z$ gives
    \begin{equation}
        \label{eq:proof:tv_prox_synthesis}
        \min_{\hat z} \frac12 \|x  - L\hat z\|_2 + \mu\|R \hat z\|_1
        \enspace .
    \end{equation}
    We denote by $z(x, \mu)$ the minimizer of the previous equation.
    Thus, the solution $u(x, \mu)$ of the original problem \autoref{eq:proof:tv_loss} can be recovered using $u(, \mu) = Lz(x, \mu)$.
    Iterative PGD can be used to solve \autoref{eq:proof:tv_prox_synthesis} and $z(x, mu)$ is a fixed point of the iterative procedure. That is to say the solution $z$ verifies
    \begin{equation}
        \label{eq:fixed_point}
        \begin{cases}
            z_1(x, \mu) = & z_1(x, \mu) - \frac{1}{\rho}(L^\top(Lz(x, \mu) - x))_1 \enspace ,\\
            z_i(x, \mu) = &
                \ST\left(z_i(x, \mu)
                    - \frac{1}{\rho}(L^\top(Lz(x, \mu) - x))_i,
                    \frac{\mu}{\rho}\right)
            \quad \text{for } i=2\dots k\enspace .
        \end{cases}
    \end{equation}
    Using the result from \autoref{lemma:st_weak_diff}, we can differentiate \autoref{eq:fixed_point} and obtain the following equation for the weak Jacobian $\widehat J_x(x, \mu) = \frac{\partial z}{\partial x}(x, \mu)$ of $z(x, \mu)$ relative to $x$
    \begin{equation}
        \widehat J_x(x, \mu) = \begin{pmatrix}
            1\\
            \ind_{\{|z_2(x, \mu)| > 0 \}}\\
            \vdots\\
            \ind_{\{|z_k(x, \mu)| > 0 \}}
        \end{pmatrix}
        \odot \left[(\Id - \frac{1}{\rho}L^\top L) \widehat J_x(x, \mu)
                    + \frac{1}{\rho}L^\top \Id\right]
        \enspace .
    \end{equation}
    Identifying the non-zero coefficient in the indicator vectors yields
    \begin{equation}
         \left\{\begin{array}{lll}
        \widehat J_{x,\mathcal S^c}(x, \mu) & = & 0\\
        \widehat J_{x, \mathcal S}(x, \mu)   & = &
            (\Id - \frac{1}{\rho} L_{:, \mathcal S}^\top L_{:, \mathcal S}) \widehat J_{x,\mathcal S}(x, \mu)
            + \frac{1}{\rho}L_{:, \mathcal S}^\top
        \enspace . \label{eq:diff_fix_point}
        \end{array}
        \right.
    \end{equation}
    As, $L$ is invertible, so is $ L_{:, \mathcal S}^\top L_{:, \mathcal S}$ for any support $\mathcal S$ and solving the second equations yields the following
    \begin{equation}
        \widehat J_{x, \mathcal S} = (L_{:, \mathcal S}^\top L_{:, \mathcal S})^{-1} L_{:, \mathcal S}^\top
    \end{equation}

    Using $u = Lz$ and the chain rules yields the expecting result for the weak  Jacobian relative to $x$, noticing that as $\hat J_{x, \mathcal S^c} = 0$, $L \hat J_x = L_{:, \mathcal S} \hat J_{x, \mathcal S}$.\\

    Similarly, concerning, $\widehat J_{\mu}(x, \mu)$, we use the result from \autoref{lemma:st_weak_diff} an differentiale \autoref{eq:fixed_point} and obtain $\widehat J_{\mu}(x, \mu) = \frac{\partial z}{\partial \mu}(x, \mu)$ of $z(x, \mu)$ relative to $\mu$

    \begin{equation}
        \widehat J_{\mu}(x, \mu) = \begin{pmatrix}
            1\\
            \ind_{\{|z_2(x, \mu)| > 0 \}}\\
            \vdots\\
            \ind_{\{|z_k(x, \mu)| > 0 \}}
        \end{pmatrix}
        \odot \left[(\Id - \frac{1}{\rho}L^\top L) \widehat J_{\mu}(x, \mu) \right] + \frac{1}{\rho}
            \begin{pmatrix}
            1\\
            - \sign(z_2(x, \mu)) \ind_{\{|z_2(x, \mu)| > 0 \}}\\
            \vdots\\
            - \sign(z_k(x, \mu)) \ind_{\{|z_k(x, \mu)| > 0 \}}
        \end{pmatrix}
        \enspace .
    \end{equation}

    Identifying the non-zero coefficient in the indicator vectors yields

    \begin{equation}
         \left\{\begin{array}{lll}
        \widehat J_{\mu,\mathcal S^c}(x, \mu) & = & 0\\
        \widehat J_{\mu, \mathcal S}(x, \mu)   & = &
            \widehat J_{\mu,\mathcal S^c}(x, \mu) -
            \frac{1}{\rho} L_{:, \mathcal S}^\top L_{:, \mathcal S}\widehat J_{\mu,\mathcal S^c}(x, \mu) - \frac{1}{\rho}\sign(z_S(x, \mu))
        \enspace .
        \label{eq:diff_mu_fix_point}
        \end{array}
        \right.
    \end{equation}

    As previous, solving the second equation yields the following

    \begin{equation}
        \widehat J_{\mu, \mathcal S} = - (L_{:, \mathcal S}^\top L_{:, \mathcal S})^{-1} \sign(z_S(x, \mu))
    \end{equation}

    Using $u = Lz$ and the chain rules yields the expecting result for the weak  Jacobian relative to $\mu$, noticing that as $\hat J_{\mu, \mathcal S^c} = 0$. \\

\end{proof}

\subsection{Convergence of the weak Jacobian}

\begin{proposition}{Linear convergence of the weak Jacobian}
    \label{prop:cvg_jacobian}
    We consider the mapping $z^{(T_in)}: , \mu\Rset^k\times \Rset_+ \mapsto \Rset^k$ defined where $z^{(T_in)}(x)$ is defined by recursion
    \begin{equation}
        \label{eq:mapping_lista}
        z^{(t)}(x, \mu) = ST(z^{(t-1)}(x, \mu) - \frac{1}{\|L\|_2^2}L^\top(Lz^{(t-1)}(x, \mu) - x), \frac{\mu}{\|L\|_2^2}
        \enspace .
    \end{equation}
    Then the weak $\mathcal J_x = L\frac{\partial z^{(T_in)}}{\partial x}$ and $\mathcal J_\mu = L\frac{\partial z^{(T_in)}}{\partial \mu}$ of this mapping relative to the inputs $x$ and $\mu$ converges linearly toward the weak Jacobian $J_x$ and $J_\mu$ of $\prox_{\mu\|\cdot\|_{TV}}(x)$ defined in \autoref{prop:jacobian_prox_tv}.
\end{proposition}

This mapping defined in \autoref{eq:mapping_lista} corresponds to the inner network in LPGD-LISTA when the weights of the network have not been learned.

\begin{proof}
    As $L$ is invertible, problem \autoref{eq:proof:tv_prox_synthesis} is strongly convex and have a unique solution.
    We can thus apply the result from \citet[Proposition~2]{Bertrand2020} which shows the linear convergence
    of the weak Jacobian $\hat{\mathcal J}_x = \frac{\partial z^{(T_in)}}{\partial x}$ and
    $\hat{\mathcal J}_\mu = \frac{\partial z^{(T_in)}}{\partial \mu}$ for ISTA toward $\hat J_x$ and
    $\hat J_\mu$ of the synthesis formulation of the prox.
    Using the linear relationship between the analysis and the synthesis formulations yields the expected result.
\end{proof}

\subsection{Estimating \texorpdfstring{$T_{in}$}{Tin} and \texorpdfstring{$T$}{T} to achieve \texorpdfstring{$\delta$}{given} error}

Using inexact proximal gradient descent results from \citet{Schmidt2011} and \citet{Machart2012}, we compute the scaling of $T_{in}$ and $T$ to achieve a given error level $\delta > 0$.

\begin{restatable}{proposition}{TinLowerBound}[Scaling of $T$ and $T_{in}$ \wrt{} the error level $\delta$]
    \label{prop:T_in_lower_bound}
    Let $\delta$ the error defined such as $P_x(u^{(T)}) - P_x(u^*) \le \delta$. \\
    We suppose there exists some constants $C_0 \ge \|u^{(0)} - u^*\|_2$ and $C_1 \ge \max_\ell\|u^{(\ell)} - \prox_{\frac{\mu}{\rho}}(u^{(\ell)})\|_2$ \\
    Then, $T$ the number of layers for the global network and $T_{in}$ the inner number of layers for the \proxTV{} scale are given by
    \[
    T_{in} = \frac{\ln{\frac{1}{\delta}} + \ln{6\sqrt{2\rho}C_1}}{\ln{\frac{1}{1-\gamma}}}
    \quad \text{and} \quad T = \frac{2\rho C_0^2}{\delta}
    \]
    with $\rho$ defined as in \autoref{eq:analysis_tv_PGD}
\end{restatable}

\begin{proof}
As discussed by \citet{Machart2012}, the global convergence rate of inexact PGD with $T_{in}$ inner iteration is given by
\begin{equation}
    \label{eq:ipdg:bound}
    \begin{split}
    & P_x(u^{(T)}) - P_x(u^*) \le \\
    & \hskip5em\frac{\rho}{2T}\left( \|u^{(0)} - u^*\|_2 + 3 \sum_{\ell=1}^T\sqrt{\frac{2(1 - \gamma)^{T_{in}}\|u^{(\ell-1)} - \prox_{\frac{\mu}{\rho}}(u^{(\ell-1)})\|_2^2}{\rho}}\right)^2
    \end{split}
    \enspace ,
\end{equation}
where $\gamma$ is the condition number for $L$ \ie $\frac{\cos(\frac{\pi}{2k+1})}{\sin(\frac{\pi}{2k+1})}$.

We are looking for minimal parameters $T$ and $T_{in}$ such that the error bound in \autoref{eq:ipdg:bound} is bellow a certain error level $\delta$.

We consider the case where there exists some constants $C_0 \ge \|u^{(0)} - u^*\|_2$ and $C_1 \ge \max_\ell\|u^{(\ell)} - \prox_{\frac{\mu}{\rho}}(u^{(\ell)})\|_2$ upper bounding how far the initialization can be compared to the result of the global problem and the sub-problems respectively.\\
We denote $\alpha_1=3\sqrt{\frac{2}{\rho}}C_1$.
The right hand side of \autoref{eq:ipdg:bound} can be upper bounded by as
\begin{equation}
    \begin{split}
        &\frac{\rho}{2T}\left(\|u^{(0)} - u^*\|_2 + 3 \sum_{\ell=1}^T\sqrt{\frac{2(1 - \gamma)^{T_{in}}\|u^{(\ell-1)} - \prox_{\frac{\mu}{\rho}}(u^{(\ell-1)})\|_2^2}{\rho}}\right)^2\\
    &\hskip20em
    \le
    \frac{\rho}{2T}\left(C_0 + \alpha_1T(1 - \gamma)^{T_{in}/2}\right)^2
\end{split}
\end{equation}
Then, we are looking for $T, T_{in}$ such that this upper bound is lower than $\delta$, \ie{}
\begin{align}
    & ~\frac{\rho}{2T}\left(C_0 + \alpha_1T(1 - \gamma)^{T_{in}/2}\right)^2 \le \delta\\
    \Leftrightarrow &
        \left(C_0 + \alpha_1T(1 - \gamma)^{T_{in}/2}\right)^2
              - \frac{2\delta}{\rho}T \le 0\\
    \Leftrightarrow &
        \left(C_0 + \alpha_1T(1 - \gamma)^{T_{in}/2}
              - \sqrt{\frac{2\delta}{\rho}}\sqrt{T}\right)
        \underbrace{\left(B + \alpha_1T(1 - \gamma)^{T_{in}/2}
              + \sqrt{\frac{2\delta}{\rho}}\sqrt{T}\right)}_{\ge 0}\le 0\\
    \Leftrightarrow &
        C_0 + \alpha_1T(1 - \gamma)^{T_{in}/2}
              - \sqrt{\frac{2\delta}{\rho}}\sqrt{T} \le 0\\
\end{align}
Denoting $\alpha_2 = \sqrt{\frac{2\delta}{\rho}}$ and $X = \sqrt{T}$, we get the following function of $X$ and $T_{in}$
\begin{align}
    f(X, T_{in})
        & = \alpha_1(1 - \gamma)^{T_{in}/ 2}X^2 - \alpha_2 X + C_0
\end{align}
The inequality $f(X, T_{in}) \le 0$ has a solution if and only if $\alpha_2^2 - 4C_0\alpha_1(1-\gamma)^{T_{in} / 2} \ge 0$ \ie{}
\[
    T_{in} \ge 2\frac{\ln{\frac{\alpha_2^2}{4\alpha_1C_0}}}{\ln{1 - \gamma}}
\]
Taking the minimal value for $T_{in}$ \ie{} $T_{in} = 2\frac{\ln{\frac{\alpha_2^2}{4\alpha_1C_0}}}{\ln{1 - \gamma}} = \frac{\ln{\frac{1}{\delta}} + \ln{6\sqrt{2\rho}C1}}{\ln{\frac{1}{1-\gamma}}} $ yields
\[
    f(X, T_{in}) = \frac{\alpha_2^2}{4C_0}X^2 - \alpha_2X + C_0
        = \frac{\alpha_2^2}{4C_0}(X - \frac{2C_0}{\alpha_2})^2
\]
for $X = \frac{2C_0}{\alpha_2} = \frac{\sqrt{2\rho} C_0}{\sqrt{\delta}}$ \ie $T = \frac{2\rho C_0^2}{\delta}$.
\end{proof}

\end{document}